\documentclass[
a4paper,
11pt, 
DIV=12,
toc=bibliography, 
headsepline, 
parskip=half-,
abstract=true,
]{scrartcl}

\usepackage[english]{babel}
\usepackage[utf8]{inputenc}
\usepackage[T1]{fontenc}
\usepackage[english,cleanlook]{isodate}
\usepackage{lmodern}
\usepackage[babel=true]{microtype}
\usepackage{amsmath}
\usepackage{mathtools}
\usepackage{amssymb}
\usepackage{amsthm}
\usepackage{mathrsfs}
\usepackage{scrlayer-scrpage} 
\usepackage{enumitem}
\usepackage{tikz} 
\usepackage{tikz-3dplot} 
\usetikzlibrary{plotmarks,intersections,arrows.meta,cd} 

\ihead{G.\,Franz, M.\,B.\,Schulz}
\ohead{\csname @title\endcsname}
\setkomafont{pageheadfoot}{\footnotesize}

\providecommand{\R}{\mathbb{R}}
\providecommand{\Z}{\mathbb{Z}}
\providecommand{\N}{\mathbb{N}}
\providecommand{\B}{\mathbb{B}}
\providecommand{\Sp}{\mathbb{S}}
\providecommand{\dih}{\mathbb{D}}
\providecommand{\pri}{\mathbb{P}}
\providecommand{\st}{\, :\ }
\providecommand{\rotation}{\textsf{R}}
\providecommand{\fbmsgO}{Z}
\providecommand{\hsd}{\mathscr{H}}
\providecommand{\gsum}{\mathfrak{g}}  
\providecommand{\bsum}{\mathfrak{b}} 
\renewcommand{\complement}{\mathsf{c}}
\providecommand{\Is}{\mathfrak{Is}}
\providecommand{\compgamma}{K}
\newcommand{\mres}{\mathbin{\vrule height 1.6ex depth 0pt width
0.13ex\vrule height 0.13ex depth 0pt width 1.3ex}}

\DeclarePairedDelimiter\abs{\lvert}{\rvert}
\DeclarePairedDelimiter\nm{\lVert}{\rVert}

\DeclarePairedDelimiter\interval{]}{[}
\DeclarePairedDelimiter\Interval{[}{[}
\DeclarePairedDelimiter\intervaL{]}{]}

\DeclareMathOperator{\dist}{dist}

\DeclareMathOperator{\genus}{genus}
\DeclareMathOperator{\spt}{supp}

\pgfmathsetmacro{\picscale}{0.39}
\pgfmathsetmacro{\unitscale}{\picscale*\textwidth/2cm}
\providecommand{\FBMS}[2][\picscale]{\draw(0,0,0)node(FBMS)[inner sep=0]{\includegraphics[width=#1\textwidth]{#2}};}
\pgfmathsetmacro{\thetaO}{72}
\pgfmathsetmacro{\phiO}{90-45/2}
\newenvironment{unitball}[1][\picscale]{%
\tdplotsetmaincoords{\thetaO}{\phiO}
\begin{tikzpicture}[scale={#1*\textwidth/2cm},tdplot_main_coords,line cap=round,line join=round,semithick,baseline={(0,0,0)}]
\pgfmathsetmacro{\laengeA}{-135-45/2+\phiO}
\pgfmathsetmacro{\laengeB}{\laengeA+45}
\begin{scope}[black!30]
\path[tdplot_screen_coords,inner color=white,outer color=cyan!5!black!20](0,0)circle(1);
\foreach\laengengrad in {\laengeA,\laengeB,...,\phiO}{
\tdplotsetthetaplanecoords{\laengengrad}
\pgfmathsetmacro{\start}{-atan( 1/(tan(\thetaO)*( cos(\laengengrad)*sin(\phiO)-cos(\phiO)*sin(\laengengrad))))}
\tdplotdrawarc[tdplot_rotated_coords]{(0,0,0)}{1}{180+\start}{360+\start}{}{}
}
\foreach\breitengrad in {-60,-30,...,60}{
\pgfmathsetmacro{\breitenkreisradius}{cos(\breitengrad)}
\pgfmathsetmacro{\hoehe}{sin(\breitengrad)}\tdplotdrawarc{(0,0,\hoehe)}{\breitenkreisradius}{0}{360}{}{}
}
\end{scope}
}{\begin{scope}[black]
\foreach\laengengrad in {\laengeA,\laengeB,...,\phiO}{
\tdplotsetthetaplanecoords{\laengengrad}
\pgfmathsetmacro{\start}{-atan( 1/(tan(\thetaO)*( cos(\laengengrad)*sin(\phiO)-cos(\phiO)*sin(\laengengrad))))}
\tdplotdrawarc[tdplot_rotated_coords]{(0,0,0)}{1}{\start}{180+\start}{}{}
} 
\foreach\breitengrad in {-60,-30,...,60}{
\pgfmathsetmacro{\breitenkreisradius}{1*cos(\breitengrad)}
\pgfmathsetmacro{\hoehe}{1*sin(\breitengrad)}
\pgfmathsetmacro{\visiblerange}{acos(-sign(\breitengrad)*min(1,abs(tan(\breitengrad)*tan(90-\thetaO)))  )}
\ifdim\visiblerange pt>0pt
\tdplotdrawarc{(0,0,\hoehe)}{\breitenkreisradius}{\phiO-90+\visiblerange}{\phiO-90-\visiblerange}{}{}
\fi
}
\draw[tdplot_screen_coords] (0,0) circle (1);
\end{scope}%
\end{tikzpicture}%
}
\makeatletter
\tikzset{
    scale plot marks/.is choice,
    scale plot marks/true/.style={},	
    scale plot marks/false/.code={
        \def\pgfuseplotmark##1{\pgftransformresetnontranslations\csname pgf@plot@mark@##1\endcsname}
    },
every mark/.append style={scale plot marks=false},
bullet/.style={mark=*,mark size=1.125pt},
}
\makeatother
\definecolor{FarbeA}{cmyk}{0,0.5,1,0}
\definecolor{FarbeB}{cmyk}{1,0.5,0,0}

\usepackage[initials, alphabetic]{amsrefs}
\makeatletter
\renewcommand{\amsrefs@warning}[1]{}  
\makeatother
 
\usepackage[
pdftitle={Topological control for min-max free boundary minimal surfaces},
pdfauthor={Giada Franz and Mario B. Schulz}, 
pdfborder={0 0 0},
]{hyperref}

\theoremstyle{plain}
\newtheorem{theorem}{Theorem}[section]
\newtheorem{lemma}[theorem]{Lemma}
\newtheorem{corollary}[theorem]{Corollary} 
\newtheorem{proposition}[theorem]{Proposition}

\theoremstyle{definition}
\newtheorem{definition}[theorem]{Definition}
 
\theoremstyle{remark}
\newtheorem{remark}[theorem]{Remark}

\title{Topological control for min-max free~boundary~minimal surfaces}
\author{Giada Franz and Mario B. Schulz}
\date{\vspace*{-2ex}}
 
\newcommand\printaddress{{
\setlength{\parindent}{17pt} 
\medskip 
\par
{\scshape Giada Franz}
\newline
CNRS and Université Gustave Eiffel, LAMA, 77420 Champs-sur-Marne, France
\newline
\textit{E-mail address:} 
\texttt{giada.franz@univ-eiffel.fr}, \texttt{giada.franz@cnrs.fr}
\par\medskip
{\scshape Mario B. Schulz}
\newline 
Università di Trento, 
Dipartimento di Matematica, 
via Sommarive 14, 
38123 Povo, 
Italy
\newline
\textit{E-mail address:} 
\texttt{mario.schulz@unitn.it}
\par
}} 

\begin{document}

\maketitle

\begin{abstract}
We establish general bounds on the topology of free boundary minimal surfaces obtained via min-max methods in compact, three-dimensional ambient manifolds with mean convex boundary. 
We prove that the first Betti number is lower semicontinuous along min-max sequences converging in the sense of varifolds to free boundary minimal surfaces. 
In the orientable case, we obtain an even stronger result which implies that if the number of boundary components increases in the varifold limit, then the genus decreases at least as much. 
We also present several compelling applications, such as the variational construction of a free boundary minimal trinoid in the Euclidean unit ball. 
\end{abstract} 

\section{Introduction}\label{sec:intro} 

Free boundary minimal surfaces are critical points for the area functional among all surfaces whose boundaries are constrained to the boundary of an ambient Riemannian manifold $M$. 
Equivalently, a free boundary minimal surface has vanishing mean curvature and meets $\partial M$ orthogonally along its own boundary. 
The case where the ambient manifold is the Euclidean unit ball is particularly interesting due to the connection with the optimization problem for the first Steklov eigenvalue 
(cf.~\cite{FraserSchoen2011, FraserSchoen2016, KarpukhinKokarevPolterovich2014, GirouardLagace2021}).
 
Constructing examples of embedded free boundary minimal surfaces is a challenging problem, especially in ambient manifolds (such as the Euclidean unit ball) which only allow unstable solutions. 
Among the most successful approaches towards existence results are gluing techniques and min-max methods. 
Gluing methods provide straightforward control on the topology and symmetry of solutions, but usually require the genus or the number of boundary components to be sufficiently large 
(cf.~\cite{FolhaPacardZolotareva2017, KapouleasLi2017, KapouleasMcGrath2020, KapouleasWiygul2017, KapouleasZou2021, CSWnonuniqueness}). 
With min-max methods one can construct examples of embedded free boundary minimal surfaces with low genus and few boundary components, as demonstrated e.\,g.~in \cite{GruterJost1986,Ketover2016FBMS,Li2015,CarlottoFranzSchulz2022}. 
However, controlling the topology of the limit surface is notoriously difficult because min-max sequences converge, a priori, only in the sense of varifolds. 

Within the context of the min-max theory established by Almgren--Pitts \cite{Almgren1965, Pitts1981} and Marques--Neves \cite{MarquesNeves2014, MarquesNeves2017}, direct control on the topology of the limit surface is out of reach so far. 
In fact, when aiming to construct free boundary minimal surfaces with prescribed topology, it is advantageous to employ the min-max theory by Simon--Smith \cite{Smith1982} and Colding--De Lellis \cite{ColdingDeLellis2003} because this variant imposes stricter criteria on the regularity and convergence of the sweepouts (see \cite[§\,2.11]{MarquesNeves2014}).
In this setting, general genus bounds have been obtained in \cite[Theorem~0.6]{DeLellisPellandini2010}, \cite[Theorem~9.1]{Li2015} and \cite{Ketover2019}. 
However, there appears to be no result that provides general control on the number of boundary components of a free boundary minimal surface obtained through min-max methods.
Li \cite[§\,9]{Li2015} has even noted the impossibility of obtaining such a bound -- at least without any assumption on the ambient manifold. 
In this article, we address this gap in the literature by presenting a result that holds in ambient manifolds with strictly mean convex boundary. 
This setting has the advantage that embedded free boundary minimal surfaces are necessarily properly embedded. 
We prove in Theorem~\ref{thm:main} that the first Betti number is lower semicontinuous along min-max sequences converging to free boundary minimal surfaces. 
Moreover, we introduce the notions of genus complexity and boundary complexity (which are closely related to the genus and the number of boundary components respectively) and prove in Theorem~\ref{thm:mainorientable} that their sum is lower semicontinuous along min-max sequences consisting of orientable surfaces.  
A key ingredient in the proof is Simon's lifting lemma, as well as a novel free boundary version of that result, Lemma~\ref{lem:BoundarySimonLemma}. 

We also present several applications. 
In Section \ref{sec:applications}, we first provide a simple, direct proof for the existence of free boundary minimal discs in convex bodies shown in \cite{GruterJost1986} and then discuss the application of our main estimate in the context of the construction of the free boundary minimal surfaces with connected boundary and arbitrary genus obtained in \cite{CarlottoFranzSchulz2022}.
Besides, we use our result to prove that there exists a free boundary minimal surface with Morse index equal to $5$ in the Euclidean unit ball, refining \cite[Theorem~1.4]{ChunPongChu2022}. 
Last but not least, Section \ref{sec:genus0} contains a variational construction of embedded free boundary minimal surfaces in the three-dimensional Euclidean unit ball with genus zero and an arbitrary number of pairwise isometric boundary components which are aligned along the equator (cf. \cite{FolhaPacardZolotareva2017,Ketover2016FBMS}).

\subsection{Notation and convention}

Throughout the article, we consider $M$ to be a compact, three-dimensional Riemannian manifold with strictly mean convex boundary $\partial M$. 
We refer to $M$ as the \emph{ambient} manifold. 
In this context, $G$ denotes a finite group of orientation-preserving isometries of $M$. 
A subset $\Sigma\subset M$ is called \emph{$G$-equivariant} if the action of $G$ on $M$ restricts to an action on $\Sigma$ which commutes with the inclusion map of $\Sigma$ into $M$. 

A submanifold $\Sigma\subset M$ is called \emph{properly embedded}, if it is embedded and 
satisfies $\partial\Sigma=\Sigma\cap\partial M$. 
Throughout this article, we call two submanifolds of $M$ \emph{isometric} if there exists an ambient isometry mapping one onto the other. 
(In the literature, the term ``congruent'' is occasionally used interchangeably.)  
A compact manifold without boundary is called \emph{closed}. 

The $k$-dimensional Hausdorff measure on $M$ is denoted by $\hsd^k$.
In particular, the area of a smooth surface $\Sigma\subset M$ is given by $\hsd^2(\Sigma)$. 
The $k$th Betti number of any topological space $X$, i.\,e.~the rank of the $k$th homology group $H_k(X)$ is denoted by $\beta_k(X)$.  

Our work relies on the theory of varifolds for which we refer the reader e.\,g.~to Chapters 4 and~8 of \cite{Simon1983} (see also \cite[§\,2]{ColdingDeLellis2003} and \cite[§\,1]{DeLellisPellandini2010} for a brief introduction).   
We only consider $2$-varifolds in $M$. 
The \emph{support} $\spt\Gamma$ of a varifold $\Gamma$ is the smallest closed subset of $M$ outside which the mass measure $\nm{\Gamma}$ vanishes identically. 
The $\varepsilon$-neighbourhood around any subset $N\subset M$ is denoted by 
$U_\varepsilon N$. 
Given a varifold $\Gamma$ in $M$ we abuse notation slightly by defining 
\begin{align}\label{eqn:Ueps}
U_{\varepsilon}\Gamma\vcentcolon=\{p\in M\st\dist_{M}(p,\spt\Gamma)<\varepsilon\}.
\end{align} 
More generally, $f(\Gamma)$ shall be understood as short-hand notation for $f(\spt\Gamma)$ whenever the latter is well-defined, for example when $\Gamma$ is induced by a smooth surface and $f=\genus$. 
In particular, this notation does not take multiplicity into account, as $f(m\Gamma)=f(\Gamma)$ for any positive integer $m$.

\subsection{Equivariant min-max theory}

We briefly recall the min-max theory à la Simon--Smith \cite{Smith1982} and Colding--De Lellis \cite{ColdingDeLellis2003}, for the equivariant, free boundary setting.
We provide the basic definitions and summarize the known results in Theorem~\ref{thm:PreviousResults}. 
The theorem is based on several contributions, among which we emphasize \cite{Smith1982,ColdingDeLellis2003,DeLellisPellandini2010,Ketover2019,Ketover2016Equivariant,Ketover2016FBMS,Li2015,Franz2021}.
For more details and comments, we refer to \cite[Part~II]{Franz2022} (where similar notation as in this paper is employed).
 
\begin{definition}[$G$-sweepout] \label{defn:G-sweepout}
Given an ambient manifold $M$ and a group $G$ of isometries as above, we say that $\{\Sigma_t\}_{t\in[0,1]^n}$ is a ($n$-parameter) $G$-sweepout of $M$ if the following properties are satisfied:
\begin{enumerate}[label={\normalfont(\roman*)}]
\item\label{defn:G-sweepout-i} $\Sigma_t$ is a $G$-equivariant subset of $M$ for all $t\in[0,1]^n$;
\item\label{defn:G-sweepout-ii} $\Sigma_t$ is a smooth, properly embedded surface in $M$ for all $t\in\interval{0,1}^n$;
\item\label{defn:G-sweepout-iii} $\Sigma_t$ varies smoothly for $t\in\interval{0,1}^n$ and continuously, in the sense of varifolds, for $t\in[0,1]^n$.
\end{enumerate}
\end{definition}

\begin{remark} \label{rmk:FiniteBadPoints}
One can relax Definition \ref{defn:G-sweepout} slightly by allowing finite sets of points in $M$ and parameters in $[0,1]^n$ where the smoothness in \ref{defn:G-sweepout-ii} and \ref{defn:G-sweepout-iii} is  not satisfied. 
With this more general definition, which is given explicitly in \cite[Definition~9.1.4]{Franz2022}, all the results in this paper remain true. 
\end{remark}

\begin{definition}\label{def:SaturationWidth}
Given a $G$-sweepout $\{\Sigma_t\}_{t\in[0,1]^n}$ of $M$, we define its \emph{$G$-saturation} $\Pi$ to be the set of all 
$\{\Phi(t,\Sigma_t)\}_{t\in[0,1]^n}$, where 
$\Phi\colon[0,1]^n\times M\to M$ is a smooth map such that $\Phi(t,\cdot)$ is a diffeomorphism which commutes with the $G$-action for all $t\in[0,1]^n$ and coincides with the identity for all $t\in\partial[0,1]^n$. 
The \emph{min-max width} of $\Pi$ is then defined as
\[
W_\Pi\vcentcolon=\adjustlimits\inf_{\{\Lambda_t\}\in \Pi~}\sup_{t\in[0,1]^n}\hsd^2({\Lambda_t}).
\]
If a sequence $\{\{\Lambda_t^j\}_{t\in[0,1]^n}\}_{j\in\N}$ in $\Pi$ is minimizing in the sense that 
$\sup_{t\in[0,1]^n}\hsd^2({\Lambda_t^j})\to W_\Pi$ as $j\to\infty$ and if $\{t_j\}_{j\in\N}$ is a sequence in $[0,1]^n$ such that $\hsd^2({\Lambda_{t_j}^j})\to W_\Pi$ as $j\to\infty$, then we call 
$\{\Lambda_{t_j}^j\}_{j\in\N}$ a \emph{min-max sequence}. 
\end{definition}

Below, we recall the equivariant min-max theorem in the free boundary setting (cf.~\cite[Theorem~9.2.1]{Franz2022} for a reference with the same notation).
As aforementioned, the result builds upon several contributions, among which we would like to mention the foundational papers \cite{Smith1982} and \cite{ColdingDeLellis2003}, the results about the lower semicontinuity of the genus in \cite{DeLellisPellandini2010,Ketover2019}, the adaptations to the free boundary setting in \cite{Li2015} and to the equivariant setting in \cite{Ketover2016Equivariant,Ketover2016FBMS}, and the estimate on the equivariant index in \cite{Franz2021}. 
We would like to remark that the statement \cite[Lemma~3.5]{Li2015} attributed to \cite{Gruter1986regularity} requires the additional assumption that $M$ has mean convex boundary, which is satisfied in our setting. 
In fact, under this assumption, the lemma follows from \cite[Theorem~6.1]{GuangLiZhou2020}.  

\begin{theorem}[Equivariant min-max]\label{thm:PreviousResults}
Let $M$ be a three-dimensional Riemannian manifold with strictly mean convex boundary and let $G$ be a finite group of orientation-preserving isometries of $M$.
Let $\{\Sigma_t\}_{t\in[0,1]^n}$ be a $G$-sweepout of $M$. 
If the min-max width $W_\Pi$ of its $G$-saturation satisfies  
\[
W_\Pi > \sup_{t\in \partial[0,1]^n}\hsd^2(\Sigma_t)
\]
then there exists a min-max sequence $\{\Sigma^j\}_{j\in\N}$ of (smooth) $G$-equivariant surfaces converging in the sense of varifolds to 
\[
\Gamma\vcentcolon= \sum_{i=1}^\compgamma m_i\Gamma_i,
\]
where the varifolds $\Gamma_1,\ldots,\Gamma_\compgamma$ are induced by pairwise disjoint, connected, embedded free boundary minimal surfaces in $M$ and where the multiplicities $m_1,\ldots,m_\compgamma$ are positive integers.
Moreover, the support of $\Gamma$ is $G$-equivariant and its $G$-equivariant index is less or equal than $n$.  
Finally, if all surfaces in the sweepout are orientable, then the genus bound
\begin{equation}\label{eqn:PreviousResults-genusbound}
\sum_{i\in\mathcal O}\genus(\Gamma_i)+\frac{1}{2}\sum_{i\in \mathcal N}\bigl(\genus(\Gamma_i)-1\bigr)
\leq\liminf_{j\to\infty}\genus(\Sigma^j) 
\end{equation}
holds. 
Here, $\mathcal O$ respectively $\mathcal N$ denote the set of indices $i\in\{1,\ldots,\compgamma\}$ such that $\Gamma_i$ is orientable respectively nonorientable.
\end{theorem}

\begin{remark}\label{rem:multiplicity}
The analogous genus bound for min-max sequences in closed ambient manifolds has been improved by Ketover \cite{Ketover2019}, taking the multiplicity of the convergence into account on the left-hand side of \eqref{eqn:PreviousResults-genusbound}. 
In this article, we focus on estimates without multiplicity because this is sufficient for all our current applications. 
Moreover, Wang--Zhou in \cite[Theorem B]{WangZhou2023} proved that, in the closed setting, multiplicity does not occur when the limit is an unstable minimal surface. This result is expected to be true in the free boundary setting as well and therefore the estimate without multiplicity would be sufficient for most applications.
\end{remark}
 
\subsection{Main results}

The goal of this paper is to improve the topological control \eqref{eqn:PreviousResults-genusbound} in Theorem~\ref{thm:PreviousResults}, 
especially in terms of the number of boundary components of the resulting free boundary minimal surface, which has been obscure until now.  
In order to state our main results Theorems \ref{thm:main} and \ref{thm:mainorientable}, we introduce a novel notion of topological complexity. 

\begin{definition}[Topological complexity] \label{defn:abc}
The \emph{genus} of any compact, connected surface $\Sigma$ is defined as the maximum number of disjoint simple closed curves which can be removed from $\Sigma$ without disconnecting it. 
Given any compact, possibly disconnected surface $\Sigma$, let $\mathcal{O}(\Sigma)$ respectively $\mathcal{N}(\Sigma)$ denote the set of its orientable respectively nonorientable connected components and let $\mathcal{B}(\Sigma)$ be the set of all connected components of $\Sigma$ with nonempty boundary. 
We define the \emph{genus complexity} $\gsum(\Sigma)$ and the \emph{boundary complexity} $\bsum(\Sigma)$ of $\Sigma$ by 
\begin{align*}
\gsum(\Sigma)&\vcentcolon=\sum_{\hat\Sigma\in\mathcal{O}(\Sigma)} \operatorname{genus}(\hat\Sigma) + \frac{1}{2}\sum_{\hat\Sigma\in\mathcal{N}(\Sigma)}\bigl(\operatorname{genus}(\hat\Sigma)-1\bigr),
\\[1ex]
\bsum(\Sigma)&\vcentcolon= \sum_{\hat\Sigma\in\mathcal{B}(\Sigma)}\bigl(\beta_0(\partial\hat\Sigma)-1\bigr)
\end{align*}
where $\beta_0(\partial\hat\Sigma)$ denotes the number of boundary components of $\hat\Sigma$ (the $0$th Betti number of~$\partial\hat\Sigma$). 
We recall that whenever $\Gamma$ is a varifold induced by a smooth surface then $\genus(\Gamma)$ is understood as the genus of the support of $\Gamma$; analogously $\gsum(\Gamma)\vcentcolon=\gsum(\spt\Gamma)$ and $\bsum(\Gamma)\vcentcolon=\bsum(\spt\Gamma)$. 
\end{definition}

\begin{remark}\label{rem:complexities}
The complexities $\gsum(\Sigma)$ and $\bsum(\Sigma)$ are nonnegative numbers for any surface $\Sigma$. 
Moreover, $\gsum$ and $\bsum$ are additive with respect to taking unions of connected components and therefore nonincreasing for the operation of discarding connected components. 
The definition of $\gsum(\Sigma)$ is consistent with the left-hand side in \cite[(0.3)]{DeLellisPellandini2010}, which is further elaborated and supported in \cite[§\,10]{DeLellisPellandini2010}. 
Corollary \ref{cor:beta1=2g+b+c} in the appendix states that the first Betti number $\beta_1(\Sigma)$ coincides with the sum of $2\gsum(\Sigma)+\bsum(\Sigma)$ and the number of nonorientable connected components of $\Sigma$ with nonempty boundary. 
\end{remark}

\begin{theorem}\label{thm:main}
In the setting of Theorem \ref{thm:PreviousResults}, where the min-max sequence $\{\Sigma^j\}_{j\in\N}$ converges in the sense of varifolds to $\Gamma$, 
the first Betti number and the genus complexity are lower semicontinuous along the min-max sequence in the sense that 
\begin{align*}
\beta_1(\Gamma)&\leq\liminf_{j\to\infty}\beta_1(\Sigma^j), &
\gsum(\Gamma)&\leq\liminf_{j\to\infty}\gsum(\Sigma^j). 
\end{align*}
\end{theorem}

If all the surfaces in the sweepout are orientable, then the lower semicontinuity of the genus complexity stated in Theorem \ref{thm:main} recovers estimate \eqref{eqn:PreviousResults-genusbound} of Theorem~\ref{thm:PreviousResults}. 
The control on the first Betti number is the main novelty in Theorem~\ref{thm:main}. 
In the case where all surfaces in the min-max sequence are orientable, 
we obtain the following sharpened version of this estimate. 
 
\begin{theorem}\label{thm:mainorientable}
In the setting of Theorem \ref{thm:PreviousResults}, if all surfaces in the min-max sequence $\{\Sigma^j\}_{j\in\N}$ are orientable, then the sum of the genus and boundary complexities is lower semicontinuous along the min-max sequence in the sense that 
\begin{align*}
\bsum(\Gamma)+\gsum(\Gamma)&\leq\liminf_{j\to\infty}\bsum(\Sigma^j)+\gsum(\Sigma^j).
\end{align*}
\end{theorem}

\pagebreak[3]

\begin{remark}\label{rem:orientable}
Properly embedded surfaces in simply connected, orientable, three-dimensional Riemannian manifolds are necessarily orientable 
(see \cite[Lemma~C.1]{ChodoshKetoverMaximo2017}). 
Therefore, Theorem~\ref{thm:mainorientable} applies when the ambient manifold $M$ is simply connected and orientable, hence, in particular, when $M$ is the three-dimensional Euclidean unit ball. 
\end{remark}

It is clear that we cannot expect the boundary complexity alone to be lower semicontinuous when passing to the limit along a min-max sequence. 
As an example one can imagine that a (hypothetical) sweepout of the Euclidean unit ball $\B^3$ consisting of annuli is given such that the corresponding min-max limit is a catenoid $\Gamma$. 
Suppose we modify each surface in the sweepout by attaching a very thin ``half-neck'' ribbon which connects the pair of boundary components. 
Any min-max sequence would then consist of surfaces $\Sigma^j$ with genus one and connected boundary. 
However, since the number of parameters has not been increased, we would still expect the min-max limit to be a catenoid, and in particular, $\gsum(\Gamma)=\gsum(\Sigma^j)-1$ and $\bsum(\Gamma)=\bsum(\Sigma^j)+1$ for all $j$. 
In this sense we expect Theorem~\ref{thm:mainorientable} to be sharp. 

A similar line of reasoning indicates that it is necessary to assume orientability of the min-max sequence in Theorem \ref{thm:mainorientable}. 
Suppose that the ambient manifold $M$ allows us to attach a thin ``twisted half-neck'' to connect the two boundary components of an annulus, in such a way that the resulting surface is properly embedded with the topology of a punctured Klein bottle. 
If we have a min-max sequence $\{\Sigma^j\}_{j\in\N}$ consisting of such surfaces, which again converges to a free boundary minimal annulus $\Gamma$ in $M$, then $\gsum(\Sigma^j)=1/2$ and $\bsum(\Sigma^j)=0$ but $\gsum(\Gamma)=0$ and $\bsum(\Gamma)=1$.  

In the hypothetical scenarios above, we modified a given sweepout by attaching a dispensable half-neck. 
In Section \ref{sec:surgery}, we rigorously define the ``inverse'' of this operation. 
The idea is to modify any given sweepout by cutting away necks and half-necks that disappear anyway when passing to the min-max limit, such that after this surgery procedure (cf.  Definition~\ref{defn:surgery}) we do have lower semicontinuity of the genus complexity and the boundary complexity independently. 
This statement is made precise in Theorem \ref{thm:TopologyOfAlmMinSeq} which serves as an intermediate step in the proof of Theorems \ref{thm:main} and \ref{thm:mainorientable}.   

We emphasize that in the aforementioned Theorem \ref{thm:TopologyOfAlmMinSeq}, an estimate on the boundary complexity is more meaningful than an estimate on the number of boundary components. 
Indeed, the surgery procedure described in Section~\ref{sec:surgery} could generate arbitrarily many spurious topological discs (with limit equal to zero in the sense of varifolds). 
These discs contribute to the number of boundary components, making this quantity a weaker upper bound, but they do not contribute to the boundary complexity since $\bsum$ vanishes for topological discs.

\paragraph{Acknowledgements.} 
The authors would like to thank Alessandro Carlotto for many helpful comments and discussions. 
This project has received funding from the Deutsche Forschungsgemeinschaft (DFG) under Germany's Excellence Strategy EXC 2044 -- 390685587, Mathematics M\"unster: Dynamics--Geometry--Structure, and the Collaborative Research Centre CRC 1442, Geometry: Deformations and Rigidity, and it has received funding from the European Research Council (ERC) under the European Union's Horizon 2020 research and innovation programme (grant agreement No.~947923). 
G.\,F.~was partially supported by NSF grant DMS-2405361. 
Finally, the authors thank the anonymous referee for valuable suggestions.

\section{Direct applications}\label{sec:applications}

In this short section, we demonstrate the strength of our result by recovering the topological control on some of the known free boundary minimal surfaces constructed via min-max methods. 

\paragraph{Discs.} 
As a first example, we choose a strictly convex subset of $\R^3$ as ambient manifold $M$ and the trivial group as $G$. 
Grüter and Jost \cite{GruterJost1986} proved that $M$ contains an embedded, free boundary minimal disc and our estimates can be applied to reproduce the topological control established in \cite[§\,5]{GruterJost1986}. 
Indeed, since the sweepout defined in \cite[§\,1 (1)--(4)]{GruterJost1986} consists of topological discs and satisfies the width estimate, Theorem \ref{thm:PreviousResults} implies the existence of some $\Gamma\vcentcolon=\sum_{i=1}^\compgamma m_i\Gamma_i$ induced by embedded free boundary minimal surfaces in $M$. 
Theorem~\ref{thm:main} then yields $\beta_1(\Gamma)=0$. 
Since $\R^3$ does not contain any closed minimal surfaces, 
we directly obtain (e.\,g.~by Proposition \ref{prop:TopCptSurf}) that each connected component $\Gamma_i$ has genus zero and connected boundary. 
In other words, $\Gamma$ is a union of embedded free boundary minimal discs.

\begin{remark} 
Our argument remains valid if the convexity assumption on $\partial M$ is relaxed to strict \emph{mean} convexity, 
provided that $M$ still allows a sweepout consisting of topological discs and satisfying the width estimate.  
\end{remark}

\paragraph{Connected boundary and arbitrary genus.} 

Our second application concerns the free boundary minimal surfaces constructed in \cite{CarlottoFranzSchulz2022}, where the ambient manifold is the Euclidean unit ball $\B^3$ and $G$ is the dihedral group $\dih_{g+1}$ (cf.~Section \ref{sec:genus0}). 
For any given $g\in\N$ there is a $\dih_{g+1}$-sweepout of $\B^3$ consisting of connected surfaces with genus $g$ and connected boundary and its $\dih_{g+1}$-saturation satisfies the width estimate (cf.~\cite[Lemma 2.2 and Corollary~3.8]{CarlottoFranzSchulz2022}). 
A corresponding min-max sequence converges in the sense of varifolds (with multiplicity one) to a connected, $\dih_{g+1}$-equivariant free boundary minimal surface $M_g\subset\B^3$ which contains the origin but is not a topological disc (cf.~\cite[Lemma 4.2 and Remark~4.1]{CarlottoFranzSchulz2022}). 
All surfaces in question are properly embedded in $\B^3$ and hence orientable by Remark \ref{rem:orientable}. 
In what follows, we compare the original control on the topology of $M_g$ with the estimates obtained by applying Theorems~\ref{thm:main} and \ref{thm:mainorientable}. 

Originally, in \cite[Lemmata 4.4 and 4.5]{CarlottoFranzSchulz2022}, the connectivity of the boundary $\partial M_g$ is established first, relying on the special structure of the chosen sweepout. 
Then the genus is determined using that $M_g$ has at least genus one because it is not a disc but has connected boundary. 

Applying Theorems~\ref{thm:main} and \ref{thm:mainorientable}, we obtain the estimates
\begin{align}\label{eqn:20230512}
\genus(M_g)=\gsum(M_g)&\leq g, & 
\bsum(M_g)+\gsum(M_g)&\leq g,
\end{align}
which do not allow us to determine the boundary complexity ahead of the genus.
Nevertheless, the recent result \cite[Lemma ~2.1~(i)]{McGrath-Zou} directly implies $\genus(M_g)\geq1$  since $M_g$ contains the origin but is not a topological disc.  
From the $\dih_{g+1}$-equivariance, we then obtain $\genus(M_g)=g$ as stated e.\,g.~in \cite[Corollary 2.10]{Buzano2025}. 
The second estimate in \eqref{eqn:20230512} now yields $\bsum(M_g)\leq 0$. 
Since $M_g$ is connected with nonempty boundary, $\beta_0(\partial M_g)=1$ follows.

\begin{remark}
With the same arguments as above, one can prove that the free boundary minimal surface $\Sigma_g$ from \cite[Theorem 1.1]{Ketover2016FBMS} has genus $g$ and at most three boundary components for all $g\in\N$. 
Only for sufficiently large $g$, it is known that $\Sigma_g$ has exactly three boundary components (cf. \cite[Appendix~D]{CSWnonuniqueness}).
\end{remark}
 
\paragraph{Index five.} 
In \cite[Theorem 1.4]{ChunPongChu2022} Chu applied (nonequivariant) min-max theory to show that the Euclidean unit ball contains an embedded free boundary minimal surface $\Gamma$ with genus $0$ or $1$ and Morse index $4$ or $5$, which is neither isometric to the equatorial disc nor the critical catenoid. 
By gaining topological control on $\Gamma$ using our results, we can show additionally that $\Gamma$ has in fact Morse index equal to $5$ as conjectured in \cite[§\,1.1]{ChunPongChu2022}. 
Indeed, (employing the relaxed definition discussed in Remark~\ref{rmk:FiniteBadPoints}) all surfaces $\Sigma_t$ in the sweepout constructed in \cite[§\,3]{ChunPongChu2022} 
satisfy $\bsum(\Sigma_t)+\gsum(\Sigma_t)\leq1$. 
Therefore, Theorem \ref{thm:mainorientable} yields $\bsum(\Gamma)+\gsum(\Gamma)\leq1$. 
In particular, since $\Gamma$ is connected, it has either genus one and connected boundary, or genus zero and two boundary components. 
(Any solution with genus zero and connected boundary is isometric to the equatorial disc by \cite{Nitsche1985}, and this case has already been excluded.) 
In either case, it follows that the Morse index of $\Gamma$ must be equal to $5$, indeed: 
\begin{itemize}
\item if $\Gamma$ has genus one then its Morse index cannot be equal to $4$ by \cite[Corollary~7.2]{Devyver2019} respectively 
\cite[Corollary~3.10]{Tran2020}
 because by construction, $\Gamma$ is embedded;
\item if $\Gamma$ has genus zero and two boundary components, then its Morse index cannot be equal to $4$ because otherwise it would be isometric to the critical catenoid by \cite[Corollary~7.3]{Devyver2019} 
respectively \cite[Corollary~3.11]{Tran2020} 
and this case has already been excluded. 
\end{itemize} 

\begin{remark}
We emphasize that the argument above relies crucially on the sharpness of the estimate for $\bsum+\gsum$ in Theorem \ref{thm:mainorientable}. 
For comparison, the bound $\beta_1(\Gamma)\leq2$ on the first Betti number (as given by Theorem \ref{thm:main}) would allow a surface $\Gamma$ with genus zero and three boundary components, and the results in 
\cite{Devyver2019} respectively \cite{Tran2020} do not exclude the possibility of such a surface having Morse index equal to $4$.

It is unknown whether the surface $\Gamma$ with index $5$ and the surface $M_1$ with genus $g=1$ and connected boundary constructed in \cite[Theorem 1.1]{CarlottoFranzSchulz2022} are isometric. 
Numerical simulations suggest that $M_1$ has indeed Morse index equal to $5$ but even then the uniqueness question remains open. 
\end{remark}

\section{Surgery}\label{sec:surgery}

In geometric topology, the concept of surgery originated from the work of Milnor \cite{Milnor1961} and refers to techniques that allow the controlled construction of a finite dimensional manifold from another. 
More specifically, surgery involves removing parts of a manifold and replacing them with corresponding parts of another manifold such that the cut matches up. 
Through surgery, we can modify the topology of a manifold while retaining certain desired properties.

\begin{definition}[{cf. Figure \ref{fig:neck} and \cite[Definition~2.2]{DeLellisPellandini2010}}]
\label{defn:surgery}
Let $\Sigma$ and $\tilde\Sigma$ be two smooth, compact surfaces which are properly embedded in the ambient manifold $M$. 
\begin{enumerate}[label={(\alph*)}]
\item\label{defn:surgery-a} 
We say that $\tilde\Sigma$ is obtained from $\Sigma$ by \emph{cutting away a neck} if  
\begin{itemize}[nosep]
\item $\Sigma\setminus\tilde\Sigma$ is homeomorphic to $\Sp^1\times\interval{0,1}$,
\item $\tilde\Sigma\setminus\Sigma$ is homeomorphic to the disjoint union of two open discs, 
\item the closure of $\Sigma\bigtriangleup\tilde\Sigma$ is a topological sphere contained in the interior of $M$.
\end{itemize}
\item\label{defn:surgery-b} 
We say that $\tilde\Sigma$ is obtained from $\Sigma$ by \emph{cutting away a half-neck} if 
\begin{itemize}[nosep]
\item $\Sigma\setminus\tilde\Sigma$ is homeomorphic to $[0,1]\times\interval{0,1}$,
\item $\tilde\Sigma\setminus\Sigma$ is homeomorphic to the disjoint union of two relatively open half-discs, i.\,e.  
sets homeomorphic to $\Interval{0,1}\times\interval{0,1}$,
\item the closure of $\Sigma\bigtriangleup\tilde\Sigma$ is topological disc which is properly embedded in $M$. 
\end{itemize}
\item\label{defn:surgery-c}
We say that $\tilde\Sigma$ is obtained from $\Sigma$ through \emph{surgery} if there is a finite number of surfaces $\Sigma=\Sigma_0,\Sigma_1,\ldots,\Sigma_N=\tilde\Sigma$ such that each $\Sigma_k$ for $k=1,\ldots,N$ is
\begin{itemize}[nosep]
\item either the union of some of the connected components of $\Sigma_{k-1}$,
\item or obtained from $\Sigma_{k-1}$ by cutting away a neck, or a half-neck, as defined above.
\end{itemize}
\end{enumerate}
\end{definition}

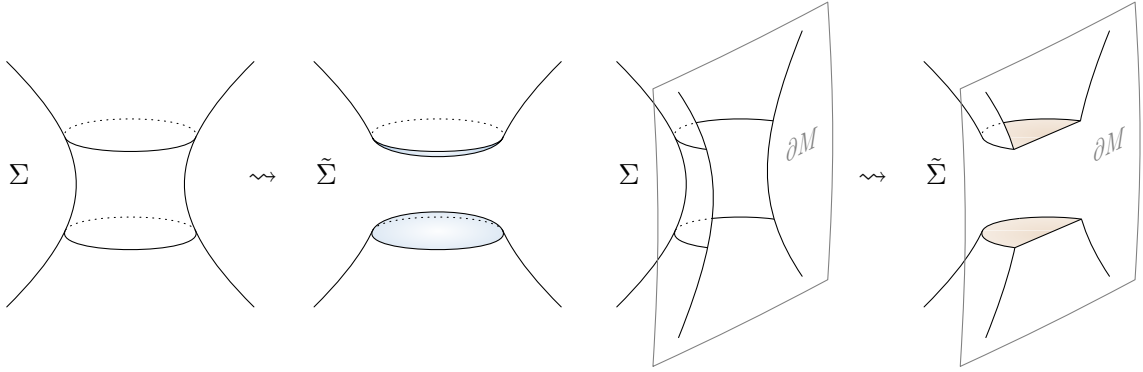
\begin{figure} 
\pgfmathsetmacro{\globalscale}{2.3}
\pgfmathsetmacro{\Radius}{1}
\pgfmathsetmacro{\angle}{45}
\pgfmathsetmacro{\speed}{\Radius*0.75}
\pgfmathsetmacro{\viewA}{0.726}
\pgfmathsetmacro{\ysl}{0.5}
\pgfmathsetmacro{\xsc}{0.5}
\pgfmathsetmacro{\viewB}{0.726*\viewA}
\tikzset{
rightbow/.style={insert path={(\angle:\Radius)..controls+(\angle:-\speed)and+(-\angle:-\speed)..(-\angle:\Radius)}},
leftbow/.style={insert path={(\angle:-\Radius)..controls+(\angle:\speed)and+(-\angle:\speed)..(-\angle:-\Radius)}},
arc4/.style={insert path={(ur.center)..controls+($\viewA*(ur.east)-\viewA*(ur.center)$)and+($\viewA*(ul.west)-\viewA*(ul.center)$)..(ul.center)}},
arc3/.style={insert path={(ur.center)..controls+($\viewB*(ur.west)-\viewB*(ur.center)$)and+($\viewB*(ul.east)-\viewB*(ul.center)$)..(ul.center)}},
arc2/.style={insert path={(dr.center)..controls+($\viewB*(dr.west)-\viewB*(dr.center)$)and+($\viewB*(dl.east)-\viewB*(dl.center)$)..(dl.center)}},
arc1/.style={insert path={(dr.center)..controls+($\viewA*(dr.east)-\viewA*(dr.center)$)and+($\viewA*(dl.west)-\viewA*(dl.center)$)..(dl.center)}},
}
\begin{tikzpicture}[line cap=round,line join=round,scale=\globalscale]
\pgfdeclarefading{fadeout}{\tikz[]{
\shade[
left color=pgftransparent!0, 
right color=pgftransparent!0, 
middle color=pgftransparent,
]
(0,0)circle(\Radius)
;}} 
\path[rightbow] 
node[pos=0.3,sloped,minimum size=1cm,inner sep=0](ur){}
node[pos=0.7,sloped,minimum size=1cm,inner sep=0](dr){}
;
\path[leftbow]
node[pos=0.3,sloped,minimum size=1cm,inner sep=0](dl){}
node[pos=0.7,sloped,minimum size=1cm,inner sep=0](ul){}
;
\begin{scope}
\clip[rightbow]arc(-\angle:-180+\angle:\Radius)[leftbow]arc(180-\angle:\angle:\Radius);
\path[inner color=FarbeB!66!black,outer color=white,path fading=fadeout](0,0)circle(\Radius);
\end{scope}
\draw[arc1][arc3];
\draw[rightbow][leftbow];
\draw[dotted][arc2][arc4];
\draw({-0.9*cos(\angle)*\Radius},0)node[anchor=base]{$\Sigma$};
\pgfresetboundingbox
\useasboundingbox[xscale=0.8](0,0)circle(\Radius);
\end{tikzpicture}
\;
\begin{tikzpicture}[line cap=round,line join=round,scale=\globalscale]
\path[rightbow] 
node[pos=0.3,sloped,minimum size=1cm,inner sep=0](ur){}
node[pos=0.7,sloped,minimum size=1cm,inner sep=0](dr){}
;
\path[leftbow]
node[pos=0.3,sloped,minimum size=1cm,inner sep=0](dl){}
node[pos=0.7,sloped,minimum size=1cm,inner sep=0](ul){}
;
\begin{scope}
\clip[rightbow]arc(-\angle:-180+\angle:\Radius)[leftbow]arc(180-\angle:\angle:\Radius);
\clip
[arc2]-|(-\Radius,-\Radius)--(\Radius,-\Radius)|-cycle
[arc3]-|(-\Radius, \Radius)--(\Radius, \Radius)|-cycle
;
\path[inner color=FarbeB!66!black,outer color=white,path fading=fadeout](0,0)circle(\Radius);
\end{scope}
\draw[inner color=FarbeB!66!black!1,outer color=FarbeB!66!black!15]
[arc1]..controls+($\viewA*(dl.east)-\viewA*(dl.center)$)and+($\viewA*(dr.west)-\viewA*(dr.center)$)..cycle;
\draw[inner color=FarbeB!66!black!1,outer color=FarbeB!66!black!30]
[arc3]..controls+($\viewA*(ul.east)-\viewA*(ul.center)$)and+($\viewA*(ur.west)-\viewA*(ur.center)$)..cycle;
\begin{scope}
\clip
[arc2]-|(-\Radius,-\Radius)--(\Radius,-\Radius)|-cycle
[arc3]-|(-\Radius, \Radius)--(\Radius, \Radius)|-cycle 
;
\draw[rightbow][leftbow];
\end{scope}
\draw[dotted][arc2][arc4];
\draw({-\Radius},0)node[anchor=base]{$\leadsto$}
({-0.9*cos(\angle)*\Radius},0)node[anchor=base]{$\tilde\Sigma$};
\pgfresetboundingbox
\useasboundingbox[xscale=0.8](0,0)circle(\Radius);
\end{tikzpicture}
\hfill
\begin{tikzpicture}[line cap=round,line join=round,scale=\globalscale]
\path[rightbow] 
node[pos=0.3,sloped,minimum size=1cm,inner sep=0](ur){}
node[pos=0.7,sloped,minimum size=1cm,inner sep=0](dr){}
;
\path[leftbow]
node[pos=0.3,sloped,minimum size=1cm,inner sep=0](dl){}
node[pos=0.7,sloped,minimum size=1cm,inner sep=0](ul){}
;
\begin{scope}
\clip[rightbow]arc(-\angle:-180+\angle:\Radius)[leftbow]arc(180-\angle:\angle:\Radius);
\clip[yslant=\ysl,xscale=\xsc][rightbow]to+(-2,0)to[bend left=120,looseness=10]cycle;
\path[inner color=FarbeA!66!black,outer color=white,path fading=fadeout](0,0)circle(\Radius);
\end{scope}
\draw[yslant=\ysl,xscale=\xsc][rightbow][leftbow];
\begin{scope}
\clip[yslant=\ysl,xscale=\xsc][leftbow]--(\angle:\Radius)[rightbow]--(\angle:-\Radius);
\draw[arc2][arc4];
\end{scope}
\begin{scope}
\clip[yslant=\ysl,xscale=\xsc][leftbow]to[bend right=90]cycle;
\draw[arc1][arc3];
\draw[dotted][arc2][arc4];
\end{scope}
\draw[leftbow];
\draw({-0.9*cos(\angle)*\Radius},0)node[anchor=base]{$\Sigma$};
\draw[yslant=\ysl,xscale=\xsc,black!50]
(-1,-0.8)to[bend right=7]
(-1, 0.8)to[bend right=1]
( 1, 0.8)to[bend left =7]
( 1,-0.8)to[bend left =1]cycle
({cos(\angle)*\Radius},0)node[yslant=\ysl,xscale=0.66,anchor=base]{$\partial M$};
\pgfresetboundingbox
\useasboundingbox[xscale=0.8](0,0)circle(\Radius);
\end{tikzpicture}
\;
\begin{tikzpicture}[line cap=round,line join=round,scale=\globalscale]
\path[name path=rbow,yslant=\ysl,xscale=\xsc][rightbow];
\path[name path=lbow,yslant=\ysl,xscale=\xsc][leftbow];
\path[name path=a1][arc1];
\path[name path=a2][arc2];
\path[name path=a3][arc3];
\path[name path=a4][arc4];
\path[
name intersections={of=rbow and a2, by=idr},
name intersections={of=rbow and a4, by=iur},
name intersections={of=lbow and a1, by=idl},
name intersections={of=lbow and a3, by=iul},
](0,0);
\path[rightbow] 
node[pos=0.3,sloped,minimum size=1cm,inner sep=0](ur){}
node[pos=0.7,sloped,minimum size=1cm,inner sep=0](dr){}
;
\path[leftbow]
node[pos=0.3,sloped,minimum size=1cm,inner sep=0](dl){}
node[pos=0.7,sloped,minimum size=1cm,inner sep=0](ul){}
;
\begin{scope}
\clip[rightbow]arc(-\angle:-180+\angle:\Radius)[leftbow]arc(180-\angle:\angle:\Radius);
\clip
[arc2]-|(-\Radius,-\Radius)--(\Radius,-\Radius)|-cycle
[arc3]-|(-\Radius, \Radius)--(\Radius, \Radius);
\clip[yslant=\ysl,xscale=\xsc][rightbow]to+(-2,0)to[bend left=120,looseness=10]cycle;
\clip(iur)--(iul)--(0,0)-|(\Radius,-\Radius)-|(-\Radius, \Radius)--(\Radius,\Radius)|-cycle;
\path[inner color=FarbeA!66!black,outer color=white,path fading=fadeout](0,0)circle(\Radius);
\end{scope}
\begin{scope}
\clip[yslant=\ysl,xscale=\xsc][leftbow]--(iur)--(iul)--cycle;
\path[outer color=FarbeA!66!black!10, inner color=FarbeA!66!black!20][arc3][arc4];
\end{scope}
\begin{scope}
\clip[arc1]-|(-\Radius,-\Radius)--(\Radius,-\Radius)|-cycle;
\draw[yslant=\ysl,xscale=\xsc][leftbow];
\end{scope}
\begin{scope}
\clip[arc2]-|(-\Radius,-\Radius)--(\Radius,-\Radius)|-cycle;
\draw[leftbow];
\draw[yslant=\ysl,xscale=\xsc][rightbow];
\end{scope}
\begin{scope}
\clip[arc3]-|(-\Radius, \Radius)--(\Radius, \Radius)|-cycle;
\draw[leftbow];
\draw[yslant=\ysl,xscale=\xsc][leftbow];
\end{scope}
\begin{scope}
\clip[arc4]-|(-\Radius, \Radius)--(\Radius, \Radius)|-cycle;
\draw[yslant=\ysl,xscale=\xsc][rightbow];
\end{scope}
\begin{scope}
\clip(idl)--(idr)--(\Radius,\Radius)-|(-\Radius,-\Radius)--cycle;
\shade[outer color=FarbeA!66!black!10, inner color=FarbeA!66!black!20][arc1][arc2];
\end{scope}
\begin{scope}
\clip[yslant=\ysl,xscale=\xsc][leftbow]to[bend right=90]cycle;
\draw[arc1][arc2][arc3];
\draw[dotted][arc4];
\end{scope}
\begin{scope}
\clip[yslant=\ysl,xscale=\xsc][leftbow]--(\angle:\Radius)[rightbow]--(\angle:-\Radius);
\draw[arc2][arc4];
\end{scope}
\draw(idl)--(idr)(iul)--(iur); 
\draw
({-\Radius},0)node[anchor=base]{$\leadsto$}
({-0.9*cos(\angle)*\Radius},0)node[anchor=base]{$\tilde\Sigma$};
\draw[yslant=\ysl,xscale=\xsc,black!50]
(-1,-0.8)to[bend right=7]
(-1, 0.8)to[bend right=1]
( 1, 0.8)to[bend left =7]
( 1,-0.8)to[bend left =1]cycle
({cos(\angle)*\Radius},0)node[yslant=\ysl,xscale=0.66,anchor=base]{$\partial M$};
\pgfresetboundingbox
\useasboundingbox[xscale=0.8](0,0)circle(\Radius);
\end{tikzpicture}
\caption{Left pair: Cutting away a neck. Right pair: Cutting away a half-neck.}%
\label{fig:neck}%
\end{figure}

\begin{lemma}\label{lem:orientable-number}
Orientability is preserved under surgery. 
By cutting away a neck or a half-neck, 
\begin{itemize}[nosep]
\item the number of connected components can increase at most by one; 
\item the number of orientable connected components can increase at most by one;
\item the number of nonorientable connected components can increase at most by one. 
\end{itemize}
\end{lemma}

\begin{proof}
Let $\tilde\Sigma$ be obtained from $\Sigma$ through surgery. 
If $\Sigma$ is orientable, then $\tilde{\Sigma}$ is orientable since any orientation on $\tilde\Sigma\cap\Sigma$ can be extended to the union $\tilde\Sigma\setminus\Sigma$ of discs respectively half-discs.  

In Definition \ref{defn:surgery} it is clear that the number of connected components can increase at most by one when cutting away a neck or a half-neck. 
For the remaining claims, we may consider without loss of generality a connected surface $\Sigma$ 
which is split into the disjoint union of two connected surfaces $\Sigma_1$ and $\Sigma_2$ by cutting away a neck or a half-neck. 
Suppose $\Sigma_1$ and $\Sigma_2$ are both orientable. 
By Definition \ref{defn:surgery}, any chosen orientation on $\Sigma\cap\Sigma_1$ can be extended consistently to $\Sigma\setminus\Sigma_2$. 
Similarly, any orientation on $\Sigma\cap\Sigma_2$ can be extended to $\Sigma\setminus\Sigma_1$. 
Since $\Sigma\setminus(\Sigma_1\cup\Sigma_2)$ is connected by Definition \ref{defn:surgery}, these two orientations either agree or disagree at every point of $\Sigma\setminus(\Sigma_1\cup\Sigma_2)$. 
If they disagree, it suffices to reverse the chosen orientation on $\Sigma\cap\Sigma_1$ in order to obtain an orientation on all of $\Sigma$. 
Therefore, $\Sigma$ is orientable, which shows that the number of orientable connected components can increase at most by one. 
Conversely, if $\Sigma_1$ and $\Sigma_2$ are both nonorientable, then $\Sigma$ is also nonorientable because surgery preserves orientability. 
This completes the proof.  
\end{proof}

The following two lemmata show that the first Betti number and the genus complexity are both nonincreasing under surgery. 
This holds without making any assumptions on the connectivity or orientability of the surfaces in question. 

\begin{lemma}\label{lem:Betti_under_surgery}
Let $\Sigma\subset M$ be any smooth, compact, properly embedded surface and let $\tilde\Sigma$ be obtained from $\Sigma$ through surgery. 
Then, $\beta_1(\tilde\Sigma)\leq\beta_1(\Sigma)$.  
\end{lemma}

\begin{proof}
Since the first Betti number is additive with respect to taking unions of connected components, it suffices to consider the case when $\tilde\Sigma$ is obtained from $\Sigma$ by cutting away a neck or a half-neck. 
The corresponding Euler characteristics then satisfy 
\begin{align}\label{eqn:EulerCharacteristic1}
\chi(\tilde\Sigma)
&=\begin{cases}
\chi(\Sigma)+2 
&\text{ if $\tilde\Sigma$ is obtained from $\Sigma$ by cutting away a neck, }
\\
\chi(\Sigma)+1  
&\text{ if $\tilde\Sigma$ is obtained from $\Sigma$ by cutting away a half-neck, }
\end{cases}
\end{align}
which follows from the fact that $\chi(X)=\chi(U)+\chi(V)-\chi(U\cap V)$ for any pair $U,V$ of open subsets covering a topological space $X$. 
We also recall (e.\,g. from \cite[Theorem 2.44]{Hatcher2002}) that the Euler characteristic coincides with the alternating sum of the Betti numbers, that is 
\begin{align}\label{eqn:EulerCharacteristic2}
\chi(\Sigma)&=\beta_0(\Sigma)-\beta_1(\Sigma)+\beta_2(\Sigma), & 
\chi(\tilde\Sigma)&=\beta_0(\tilde\Sigma)-\beta_1(\tilde\Sigma)+\beta_2(\tilde\Sigma).
\end{align}
In the case when $\tilde\Sigma$ is obtained from $\Sigma$ by cutting away a neck, 
equations \eqref{eqn:EulerCharacteristic1} and \eqref{eqn:EulerCharacteristic2} imply 
\begin{align}\label{eqn:20230620-1}
  \beta_1(\tilde\Sigma)-\beta_1(\Sigma)
&=\beta_0(\tilde\Sigma)-\beta_0(\Sigma)
 +\beta_2(\tilde\Sigma)-\beta_2(\Sigma)-2. 
\end{align}
Lemma \ref{lem:orientable-number} implies $\beta_0(\tilde\Sigma)-\beta_0(\Sigma)\leq1$ and $\beta_2(\tilde\Sigma)-\beta_2(\Sigma)\leq1$, because for the surfaces in question, the second Betti number coincides with the number of their orientable connected components without boundary.
Thus, \eqref{eqn:20230620-1} yields $\beta_1(\tilde\Sigma)-\beta_1(\Sigma)\leq0$.

In the case when $\tilde\Sigma$ is obtained from $\Sigma$ by cutting away a half-neck, equations \eqref{eqn:EulerCharacteristic1} and \eqref{eqn:EulerCharacteristic2} imply 
\begin{align}\label{eqn:20230620-2}
  \beta_1(\tilde\Sigma)-\beta_1(\Sigma)
&=\beta_0(\tilde\Sigma)-\beta_0(\Sigma)
 +\beta_2(\tilde\Sigma)-\beta_2(\Sigma)-1. 
\end{align}
Lemma \ref{lem:orientable-number} again implies $\beta_0(\tilde\Sigma)-\beta_0(\Sigma)\leq1$. 
Moreover, $\beta_2(\tilde\Sigma)-\beta_2(\Sigma)=0$ because the connected components of $\Sigma$ and $\tilde\Sigma$ which are affected by the half-neck surgery necessarily have nonempty boundary. 
Thus, \eqref{eqn:20230620-2} yields $\beta_1(\tilde\Sigma)-\beta_1(\Sigma)\leq0$ and the claim follows. 
\end{proof}

\begin{lemma}\label{lem:gsum_under_surgery}
Let $\Sigma\subset M$ be any smooth, compact, properly embedded surface and let $\tilde\Sigma$ be obtained from $\Sigma$ through surgery. 
Then, $\gsum(\tilde\Sigma)\leq\gsum(\Sigma)$.  
\end{lemma}

\begin{proof} 
Recalling Remark \ref{rem:complexities}, it again suffices to consider the case when $\tilde\Sigma$ is obtained from $\Sigma$ by cutting away a neck or a half-neck. 
Let us define the short-hand notation $\gsum=\gsum(\Sigma)$,  
\begin{align*}
b&=\text{ number of boundary components of $\Sigma$,}	\\ 
c_{\mathcal{O}}&=\text{ number of orientable connected components of $\Sigma$,}	\\ 
c_{\mathcal{N}}&=\text{ number of nonorientable connected components of $\Sigma$,}
\end{align*}
and, analogously, $\tilde\gsum$, $\tilde b$, $\tilde c_{\mathcal{O}}$, $\tilde c_{\mathcal{N}}$ for the surface $\tilde{\Sigma}$ after surgery. 
By Corollary~\ref{cor:Eulerabc}, the Euler characteristics of $\Sigma$ and $\tilde\Sigma$ are given by 
\begin{align}\label{eqn:EulerCharacteristic3}
\chi(\Sigma)&=2c_\mathcal{O}+c_\mathcal{N}-2\gsum-b, & 
\chi(\tilde\Sigma)&=2\tilde c_\mathcal{O}+\tilde c_\mathcal{N}-2\tilde\gsum-\tilde b. 
\end{align}
Moreover, Lemma \ref{lem:orientable-number} implies  
\begin{align}\label{eqn:20230618-1}
(\tilde c_{\mathcal{O}}+\tilde c_{\mathcal{N}})	- (c_{\mathcal{O}}+c_{\mathcal{N}})&\in\{0,1\}, &
\tilde c_{\mathcal{O}}-c_{\mathcal{O}}&\in\{0,1\}. 
\end{align} 
In the case when $\tilde\Sigma$ is obtained from $\Sigma$ by cutting away a neck, the boundary of $\Sigma$ is not modified. 
Hence we have $\tilde{b}=b$. 
Recalling $\chi(\tilde\Sigma)=\chi(\Sigma)+2$ from \eqref{eqn:EulerCharacteristic1} and combining this identity with \eqref{eqn:EulerCharacteristic3} and \eqref{eqn:20230618-1}, we obtain 
\begin{align}\label{eqn:2g-neckcase}
2\tilde\gsum-2\gsum&=2(\tilde c_{\mathcal{O}}-c_{\mathcal{O}})+(\tilde c_{\mathcal{N}}-c_{\mathcal{N}})-2	 
\leq (\tilde c_{\mathcal{O}}-c_{\mathcal{O}})-1\leq0. 
\intertext{In the case when $\tilde\Sigma$ is obtained from $\Sigma$ by cutting away a half-neck, we have $\chi(\tilde\Sigma)=\chi(\Sigma)+1$ by \eqref{eqn:EulerCharacteristic1} and \eqref{eqn:EulerCharacteristic3} yields  
}
\label{eqn:2g-halfneckcase}
2\tilde\gsum-2\gsum&=2(\tilde c_{\mathcal{O}}-c_{\mathcal{O}})+(\tilde c_{\mathcal{N}}-c_{\mathcal{N}})-(\tilde b-b)-1.
\end{align}
We distinguish the following cases.
\begin{enumerate}[label={(\alph*)}]
\item\label{half-neck-case1} The two boundary segments which are affected by the surgery belong to the same boundary component of $\Sigma$, in which case we have one of the following:
\begin{enumerate}[label={(a\oldstylenums{\arabic*})}]
\item\label{case1a} $\tilde{b}=b$. 
Then necessarily $\tilde c_{\mathcal{O}}+\tilde c_{\mathcal{N}}=c_{\mathcal{O}}+c_{\mathcal{N}}$ and \eqref{eqn:20230618-1} and \eqref{eqn:2g-halfneckcase} imply $\tilde\gsum\leq\gsum$. 
\item $\tilde{b}=b+1$. Then \eqref{eqn:20230618-1} and \eqref{eqn:2g-halfneckcase} imply the same estimate as in \eqref{eqn:2g-neckcase}. 
\end{enumerate}
\item\label{half-neck-case2} The two boundary segments which are affected by the surgery belong to two different boundary components of $\Sigma$, in which case we have $\tilde{b}=b-1$ and $(\tilde c_{\mathcal{O}}+\tilde c_{\mathcal{N}})=(c_{\mathcal{O}}+c_{\mathcal{N}})$.  
Then \eqref{eqn:2g-halfneckcase} implies $2\tilde\gsum-2\gsum=\tilde c_{\mathcal{O}}-c_{\mathcal{O}}$. Below we prove that in this case, we indeed have $\tilde c_{\mathcal{O}}-c_{\mathcal{O}}=0$ which completes the proof. 
\end{enumerate}
Let $\Sigma$ and $\tilde\Sigma$ be as in case \ref{half-neck-case2}. 
Then $\Sigma$ has at least two boundary components $\gamma_1$ and $\gamma_2$ forming a half-neck $N\vcentcolon=\Sigma\setminus\tilde\Sigma$ between them.  
We claim that if $\Sigma$ is nonorientable, then $\tilde\Sigma$ is nonorientable, too. 
Without loss of generality we may assume that $\Sigma$ and $\tilde\Sigma$ are both connected because cutting away a half-neck between two different boundary components does not disconnect a surface. 
Any neighbourhood $U$ of $\gamma_1$ in $\Sigma$ which is homeomorphic to $\gamma_1\times\Interval{0,1}$ is orientable. 
Moreover, the intersection $U\cap N$ is connected since $\gamma_1$ and $\gamma_2$ are disjoint. 
In particular, any chosen orientation $O$ on $U$ can be extended to a neighbourhood of the half-neck $N$.  
This implies that any orientation-reversing path in $\Sigma$ passing through $N$ can be modified such that it instead passes through $U\setminus N$ and is still orientation-reversing. 
Consequently $\tilde\Sigma$ is nonorientable as claimed. 
\end{proof}

\begin{remark}
Case \ref{case1a} in the proof of Lemma \ref{lem:gsum_under_surgery} occurs for example during half-neck surgery on a connected component which is homeomorphic to a M\"obius band, in which case we have $(\tilde c_{\mathcal{O}},\tilde c_{\mathcal{N}})=(c_{\mathcal{O}}+1,c_{\mathcal{N}}-1)$ and $\tilde\gsum=\gsum=0$. 
\end{remark}

\begin{lemma}\label{lem:orientable_surgery}
Let $\Sigma\subset M$ be any orientable, smooth, compact, properly embedded surface and let $\tilde\Sigma$ be obtained from $\Sigma$ through surgery. 
Then, $\tilde\Sigma$ is orientable and 
\begin{align*}
\bsum(\tilde\Sigma)+\gsum(\tilde\Sigma)&\leq\bsum(\Sigma)+\gsum(\Sigma).  
\end{align*}
\end{lemma}

\begin{proof}
As stated in Lemma \ref{lem:orientable-number}, orientability is preserved under surgery. 
By Remark \ref{rem:complexities}, it suffices to consider the case when $\tilde\Sigma$ is obtained from $\Sigma$ by cutting away a neck or a half-neck.  
In analogy with the proof of Lemma \ref{lem:gsum_under_surgery} we employ the short-hand notation $\gsum=\gsum(\Sigma)$, $\bsum=\bsum(\Sigma)$,  
\begin{align*}
b&=\text{ number of boundary components of $\Sigma$,}	\\ 
c&=\text{ number of connected components of $\Sigma$,}	 
\end{align*}
and, similarly, $\tilde\gsum$, $\tilde\bsum$, $\tilde b$, $\tilde c$ for the surface $\tilde{\Sigma}$ after surgery. 
In the case when $\tilde\Sigma$ is obtained from $\Sigma$ by cutting away a neck, we have $\tilde{b}=b$ because the boundary is not modified. 
If cutting away the neck adds a connected component with nonempty boundary, $\tilde\bsum=\bsum-1$ and otherwise $\tilde\bsum=\bsum$. 
In any case, $\tilde\bsum\leq\bsum$ and the claim follows directly recalling Lemma \ref{lem:gsum_under_surgery}. 

In the case when $\tilde\Sigma$ is obtained from $\Sigma$ by cutting away a half-neck, we have $\chi(\tilde\Sigma)=\chi(\Sigma)+1$ by \eqref{eqn:EulerCharacteristic1}. 
In the orientable case, Corollary~\ref{cor:Eulerabc} implies $\chi(\Sigma)=2c-2\gsum-b$ and we obtain  
\begin{align}\label{eqn:20230620-2gsum}
2\tilde\gsum-2\gsum&=2\tilde c-2c-(\tilde b-b) -1.
\shortintertext{By Definition \ref{defn:abc}, we have }
\label{eqn:20230620-bsum}
\tilde\bsum-\bsum&=(\tilde b-b)-(\tilde c-c).
\intertext{Multiplying equation \eqref{eqn:20230620-bsum} by two and adding it to \eqref{eqn:20230620-2gsum} yields}
\notag
2(\tilde\bsum+\tilde\gsum)-2(\bsum+\gsum)&=(\tilde b-b)-1.
\end{align}
Since the number of boundary components can increase at most by one when cutting away a half-neck, the claim follows. 
\end{proof}

The following proposition for surfaces with boundary is analogous to \cite[Proposition~2.3]{DeLellisPellandini2010} in the closed case. 
Roughly speaking, it states that, after suitable surgery, a sequence converging in the sense of varifolds to a limit $\Gamma$ is contained in the tubular neighbourhood $U_{2\varepsilon}\Gamma$, as defined in equation \eqref{eqn:Ueps}. 
The proof is similar to the arguments for \cite[Proposition 2.3]{DeLellisPellandini2010} and for \cite[Lemma~4.4]{CarlottoFranzSchulz2022}. 
The main difference here is the presence of the group $G$ and the possibility for half-neck surgeries.
 
\begin{proposition}\label{prop:PerformingSurgery}
Let $\{\Sigma^j\}_{j\in\N}$ be a sequence of smooth, compact, 
$G$-equivariant 
surfaces which are properly embedded in $M$ and converge in the sense of varifolds to $\Gamma = \sum_{i=1}^\compgamma m_i\Gamma_i$, where the varifolds $\Gamma_1,\ldots,\Gamma_\compgamma$ are induced by smooth, connected, pairwise disjoint surfaces.
Then, for every sufficiently small $\varepsilon>0$, there exists $J_\varepsilon\in\N$ such that for all $j\geq J_\varepsilon$ there is a 
$G$-equivariant 
surface $\tilde\Sigma^j$ obtained from $\Sigma^j$ through surgery and satisfying 
\begin{align*}
\tilde\Sigma^j&\subset U_{2\varepsilon}\Gamma, &
\tilde\Sigma^j\cap U_\varepsilon\Gamma&=\Sigma^j\cap U_\varepsilon\Gamma.
\end{align*} 
\end{proposition}

\begin{proof}
Let $\varepsilon>0$ be so small that there exists a smooth retraction of $U_{2\varepsilon}\Gamma$ onto $\spt\Gamma$. 
Convergence in the sense of varifolds implies that given any $\eta>0$ there exists $J_{\varepsilon,\eta}\in\N$ such that $\hsd^2(\Sigma^j\setminus U_\varepsilon\Gamma)<\eta$ for every $j\geq J_{\varepsilon,\eta}$.  
For $s\in\interval{0,2\varepsilon}$ the (relative) boundaries 
\begin{align*}
V_s\Gamma\vcentcolon=\overline{\partial(U_s\Gamma)\setminus\partial M}
\end{align*}
foliate $U_{2\varepsilon}\Gamma$ smoothly. 
Hence, by the coarea formula, there is a finite constant $C>0$ such that for all $j\geq J_{\varepsilon,\eta}$
\begin{align*}
\int^{2\varepsilon}_{\varepsilon}\hsd^1(\Sigma^j\cap V_s\Gamma)\,ds
&\leq C\hsd^2(\Sigma^j\setminus U_\varepsilon\Gamma)<C\eta.
\end{align*}
Thus, there exists an open subset $E\subset\interval{\varepsilon,2\varepsilon}$ with measure at least $\varepsilon/2$ such that for all $s\in E$   
\begin{align*}
	\hsd^1(\Sigma^j\cap V_s\Gamma)<\frac{2C\eta}{\varepsilon}. 
\end{align*}
Since both surfaces are compact, the set of all $s\in E$ for which $V_s\Gamma$ intersects $\Sigma^j$ transversally is open and dense in $E$ by Sard's theorem.  
Hence, there exist $s_0\in\interval{\varepsilon,2\varepsilon}$ and $\delta>0$ such that $[s_0-\delta,s_0+\delta]\subset E$ and such that 
for all $s\in[s_0-\delta,s_0+\delta]$ the intersection between $\Sigma^j$ and $V_s\Gamma$ is transverse. 
In particular, any connected component of $\Sigma^j\cap V_s\Gamma$ is either a simple closed curve or a segment connecting two points on $\partial\Sigma^j$. 

There exists $\lambda>0$ (depending on $\Gamma$ and $\varepsilon$) such that for any $s\in\interval{\varepsilon,2\varepsilon}$ any simple closed curve in $V_s\Gamma$ with length less than $\lambda$ bounds an embedded disc in $V_s\Gamma$ and the closure of any segment in $V_s\Gamma$ connecting two points on $\partial M
$ with length less than $\lambda$ bounds an embedded half-disc in $V_s\Gamma$. 
Choosing first $\eta>0$ such that $2C\eta<\lambda\varepsilon$ and then $j\geq J_{\varepsilon,\eta}$ and 
$s\in[s_0-\delta,s_0+\delta]\subset E\subset\interval{\varepsilon,2\varepsilon}$ as above, we ensure that 
each connected component $v$ of $\Sigma^j\cap V_s\Gamma$ has length less than $\lambda$ and thus bounds either a disc or a half-disc in $V_s\Gamma$ which we denote by $D_{s}^{v}$. 
In particular, 
$\Sigma^j\cap (U_{s_0+\delta}\Gamma\setminus\overline{U_{s_0-\delta}\Gamma})$  
is a finite collection of embedded, pairwise disjoint necks and half-necks in the sense of Definition~\ref{defn:surgery}. 
In principle these necks could be ``nested'' in the sense that 
$D_{s_0}^{v}\subsetneq D_{s_0}^{w}$ for different connected components $v,w$ of $\Sigma^j\cap V_{s_0}\Gamma$. 
Note that in this case $D_{s}^{v}\subsetneq D_{s}^{w}$ for all $s\in[s_0-\delta,s_0+\delta]$ since $\Sigma^j$ is embedded. 
The image of the (possibly noninjective) map $v\mapsto\hsd^2(D_{s_0}^v)$ is a discrete set of values $a_1<\ldots <a_m$.
Let $v$ be a connected component of $\Sigma^j\cap V_{s_0}\Gamma$ such that $D_{s_0}^v$ has area $a_k$. 
Removing the corresponding connected component of $\Sigma^j\cap (U_{s_0+\delta/k}\Gamma\setminus\overline{U_{s_0-\delta/k}\Gamma})$ and replacing it with $D_{s_0\pm\delta/k}^v$ is an admissible surgery in the sense of Definition~\ref{defn:surgery} provided it is done for all such $v$ and all $k\in\{1,\ldots,m\}$ in increasing order. 
This procedure preserves the $G$-equivariance, since at each step the union of all surfaces involved is $G$-equivariant.
Indeed, since $\Sigma^j$ and $V_s\Gamma$ are $G$-equivariant, the union of all discs $D_{s_0}^v$ having area $a_k$ is also $G$-equivariant. 

Let $\hat\Sigma^j$ be the new surface obtained from $\Sigma^j$ through the procedure described in the previous paragraph. 
By construction, $\hat\Sigma^j$ is disjoint from $V_{s_0}\Gamma$. 
We may regularize $\hat\Sigma^j$ such that $\hat\Sigma^j$ is smooth and still has the properties of being $G$-equivariant, disjoint from $V_{s_0}\Gamma$ and obtained from $\Sigma^j$ through surgery. 
Then, $\tilde\Sigma^j\vcentcolon=\hat\Sigma^j\cap U_{s_0}\Gamma$ is a surface obtained from $\hat\Sigma^j$ by dropping a finite number of connected components and it satisfies 
$\tilde\Sigma^j\subset U_{2\varepsilon}\Gamma$ and $\tilde\Sigma^j\cap U_\varepsilon\Gamma=\Sigma^j\cap U_\varepsilon\Gamma$.
\end{proof}

\section{Topological lower semicontinuity} \label{sec:mainproofs}

The goal of this section is to prove Theorems~\ref{thm:main} and~\ref{thm:mainorientable}.
We start by introducing further notation and terminology. 
As before, $M$ denotes the compact, three-dimensional ambient manifold with strictly mean convex boundary $\partial M$ and $G$ is a finite group of orientation-preserving isometries of $M$. 
An \emph{isotopy} of $M$ is a smooth map $\Psi\colon[0,1]\times M\to M$ such that $\Psi(s,\cdot)$ is a diffeomorphism for all $s\in[0,1]$ which coincides with the identity for $s=0$.

\begin{definition}
We say that an open set $U\subset M$ is \emph{$G$-compatible} if $\varphi(U)$ is either disjoint from $U$ or equal to $U$ for every $\varphi\in G$.
Given a $G$-equivariant surface $\Sigma\subset M$ and a $G$-compatible set $U\subset M$, we denote by $\Is_G^\delta(U,\Sigma)$ the set of isotopies $\Psi\colon[0,1]\times M\to M$ satisfying the following properties:
\begin{itemize}[nosep]
\item $\Psi$ is supported in $U$ in the sense that $\Psi(s,x)=x$ for all $x\in M\setminus U$ and $s\in[0,1]$;
\item $\hsd^2\bigl(\Psi(s,\Sigma)\bigr)\leq\hsd^2(\Sigma)+\delta$ for all $s\in[0,1]$;
\item $\Psi(s,\cdot)$ commutes with the action of all $\varphi\in G$ satisfying $\varphi(U)=U$.
\end{itemize}
\end{definition}

\begin{definition}
\label{def:AlmostMinimizing}
Given $\delta,\epsilon>0$, a $G$-compatible open set $U\subset M$, and a $G$-equivariant surface $\Sigma\subset M$, we say that $\Sigma$ is \emph{$(G,\delta,\epsilon)$-almost minimizing} in $U$ if for every isotopy $\Psi\in\Is_G^\delta(U,\Sigma)$ 
\[
\hsd^2(\Psi(1,\Sigma))\geq\hsd^2(\Sigma)-\epsilon.
\]
A sequence $\{\Sigma^j\}_{j\in\N}$ of $G$-equivariant surfaces is called \emph{$G$-almost minimizing} in $U$ if there exist $\delta_j,\epsilon_j>0$ with $\epsilon_j\to 0$ as $j\to\infty$ such that $\Sigma^j$ is $(G,\delta_j,\epsilon_j)$-almost minimizing in $U$ for all $j\in\N$.
\end{definition}

The open metric ball $B_{r}(x)\subset M$ of radius $r>0$ around any $x\in M$ is $G$-compatible provided that $r$ is sufficently small 
(cf.~\cite[Remark~10.1.3]{Franz2022}). 
While min-max sequences are not even locally area minimizing in general, they are almost minimizing in sufficiently small annuli, as stated in the subsequent lemma, where $\mathcal{AN}_{r}(x)$ denotes the collection of all annuli $B_{r_2}(x)\setminus\overline{B_{r_1}(x)}\subset M$ for some $0<r_1<r_2<r$. 

\begin{lemma}[{cf.~\cite[Lemma~13.2.4 and Proposition~13.5.3]{Franz2022}}]\label{lem:SeqAlmMin}
There exists a $G$-invariant function $r\colon M \to\interval{0,\infty}$ such that the min-max sequence $\{\Sigma^j\}_{j\in\N}$ from Theorem~\ref{thm:PreviousResults} is $G$-almost minimizing in every set in $\mathcal{AN}_{r(x)}(x)$ for all $x\in M$. 
\end{lemma}

The regularity theorem for the limit of an almost minimizing sequence is stated below. 
For the proof and the literature related to this result, we refer to \cite[Proposition~13.5.3]{Franz2022}.

\begin{proposition}
\label{prop:SmoothLimitAlmMin}
Let $M$ be a three-dimensional Riemannian manifold with strictly mean convex boundary and let $G$ be a finite group of orientation-preserving isometries of $M$. 
Let $\{\Sigma^j\}_{j\in\N}$ be a sequence of smooth surfaces which are properly embedded in $M$ and $G$-almost minimizing in every set in $\mathcal{AN}_{r(x)}(x)$ for all $x\in M$, where $r\colon M \to\interval{0,\infty}$ is a $G$-invariant function. 
Then (up to a subsequence) $\{\Sigma^j\}_{j\in\N}$ converges in the sense of varifolds to $\Gamma=\sum_{i=1}^{\compgamma}m_i\Gamma_i$, 
where the varifolds $\Gamma_1,\ldots,\Gamma_\compgamma$ are induced by pairwise disjoint, connected free boundary minimal surfaces in $M$ and where the multiplicities $m_1,\ldots,m_\compgamma$ are positive integers.
\end{proposition}

\subsection{Lifting lemmata}

Simon's lifting lemma is key for the control on the topology of the limit surface $\Gamma$ obtained in Proposition~\ref{prop:SmoothLimitAlmMin}. 
We recall the notation $U_\varepsilon\Gamma$ for the $\varepsilon$-neighbourhood around the support of $\Gamma$ in $M$ from equation \eqref{eqn:Ueps}.  
As in the proof of Proposition \ref{prop:PerformingSurgery}, there exists some $\varepsilon_0>0$ such that there is a smooth retraction of $U_{2\varepsilon_0}\Gamma$ onto the support of $\Gamma=\sum_{i=1}^{\compgamma}m_i\Gamma_i$. 
 
\begin{lemma}[{Simon's lifting lemma, cf.~\cite[Proposition~2.1]{DeLellisPellandini2010}}] \label{lem:SimonLemma}
In the setting of Proposition~\ref{prop:SmoothLimitAlmMin}, let $\gamma$ be a simple closed curve in the support of $\Gamma_i$ for some $i\in\{1,\ldots,\compgamma\}$ and let $0<\varepsilon\leq\varepsilon_0$. 
Then, for all sufficiently large $j\in\N$, there is a positive integer $m\le m_i$ and a closed curve $\gamma^j$ in $\Sigma^j\cap U_\varepsilon\Gamma_i$ which is homotopic to $m\gamma$ in $U_\varepsilon\Gamma_i$. 

\end{lemma}

\begin{remark}
Simon's lifting lemma does not require the surfaces $\Sigma^j$ to be orientable. 
Indeed, most of the proof relies on local arguments (which do not detect the orientability of the surface). 
The only step where the argument is of global nature can be found in Section~4.3 of \cite{DeLellisPellandini2010}, but one can easily check that this does not rely on orientability either (even though orientability is part of \cite[Definition~0.5]{DeLellisPellandini2010}). 
Moreover, Lemma \ref{lem:SimonLemma} is robust against relaxing the assumptions on the smoothness of the sweepout according to Remark~\ref{rmk:FiniteBadPoints} because curves can always be chosen to avoid finitely many points. 
\end{remark}

We aim at proving a free boundary version of Simon's lifting lemma in the sense that it applies to ``unclosed'' curves, or -- more generally -- a loopfree network of curves equipped with a tree structure.

\begin{definition}[Trees] 
Let $U\subseteq M$ be any submanifold. 
A \emph{tree} in $U$ consists of finite sets of vertices $\{v_0,\ldots,v_n\}\subset U$ and edges $\gamma_1,\ldots,\gamma_n\subset U$ such that 
\begin{itemize}[nosep]
\item each edge is a smooth, embedded curve in $U$ connecting two distinct vertices, 
\item the intersection of two distinct edges consists of at most one vertex, 
\item the union of all edges is connected and does not contain any closed, embedded curves. 
\end{itemize}
Vertices contained in exactly one edge are called \emph{leaves}. 
We call a tree in $U$ \emph{properly embedded} in $M$ if the set of leaves coincides with $(\bigcup_{i=1}^n\gamma_i)\cap(\partial M)$. 
A \emph{rooted} tree has one designated first vertex, e.\,g.~$v_0$, called \emph{root}. 
For our purposes it is convenient to require that the root of a tree must be an \emph{interior} point in $M$, 
ensuring that rooted trees feature at least one interior vertex. 

We call two properly embedded trees in $U$ \emph{properly homotopic} in $U$ if the two trees can be continuously deformed into each other while preserving the local structure around each vertex and each edge, and while constraining the leaves to $\partial M$. 
More precisely, if the first tree has vertices $\{v_0,\ldots,v_n\}\subset U$ and edges $\gamma_1,\ldots,\gamma_n\subset U$ then the second tree has the same number of vertices $\{v_0',\ldots,v_n'\}\subset U$ and edges $\gamma_1',\ldots,\gamma_n'\subset U$, and there exists a continuous function $H\colon[0,1]\times\bigcup_{i=1}^n\gamma_i\to U$ such that 
\begin{itemize}[nosep]
\item $H(0,\cdot)$ is the identity, $H(1,v_i)=v_i'$ for all $i=0,\ldots,n$ and $H(1,\gamma_i) = \gamma_i'$ for all $i=1,\ldots,n$; 
\item for every $t\in[0,1]$ the two sets $\{H(t,v_0),\ldots,H(t,v_n)\}$ and 
$\{H(t,\gamma_1),\ldots,H(t,\gamma_n)\}$ form a properly embedded tree in $U$. 
\end{itemize}
\end{definition}
 
\begin{lemma}[Tree lifting lemma]\label{lem:BoundarySimonLemma}
In the setting of Proposition~\ref{prop:SmoothLimitAlmMin}, let $T$ be a properly embedded, rooted tree in $\Gamma_i$ for some $i\in\{1,\ldots,\compgamma\}$. 
Given $0<\varepsilon\leq\varepsilon_0$ and any sufficiently large $j\in\N$, there is 
a properly embedded tree $T^j$ in $\Sigma^j\cap U_\varepsilon\Gamma_i$ which is properly homotopic to $T$ in $U_\varepsilon\Gamma_i$. 
\end{lemma}
 
The proof of Lemma \ref{lem:BoundarySimonLemma} relies on the local behaviour of almost minimizing sequences. 
This approach is facilitated by the following lemma. 

\begin{lemma}\label{lem:LocalPictureConv}
In the setting of Proposition~\ref{prop:SmoothLimitAlmMin}, let $B\subset M$ be a closed, $G$-compatible ball with sufficiently small radius such that there exist $\delta_j,\epsilon_j>0$ with $\epsilon_j\to0$ as $j\to\infty$ such that $\Sigma^j$ is $(G,\delta_j,\epsilon_j)$-almost minimizing in $B$. 
For every $j\in\N$, let $\{\Phi_k^j\}_{k\in\N}$ be a sequence of isotopies in $\Is_G^{\delta_j}(B,\Sigma^j)$ such that
\begin{equation}\label{eq:DefPhi}
\lim_{k\to\infty}\hsd^2\bigl(\Phi^j_k(1,\Sigma^j)\bigr)
=\inf_{\Psi\in\Is_G^{\delta_j}(B,\Sigma^j)}\hsd^2\bigl(\Psi(1,\Sigma^j)\bigr).
\end{equation}
Then, the following statements hold (cf. Figure~\ref{fig:scheme_local}). 
\begin{enumerate}[label={\normalfont(\roman*)}]
\item\label{lpc:vj} 
For all $j\in\N$ a subsequence of $\{\Phi_k^j(1,\Sigma^j)\}_{k\in\N}$ converges in the sense of varifolds in $B$ to a smooth, properly embedded, $G$-stable minimal surface $V^j\subset B$, which satisfies 
$\partial V^j\setminus\partial M=\Sigma^j\cap \partial B\setminus\partial M$ 
and the free boundary condition on $\partial M\cap B$ (if the latter is nonempty). 
\item\label{lpc:vjconv} 
The sequence $\{V^j\}_{j\in\N}$ converges in the sense of varifolds in $B$ to the same limit 
$\Gamma=\sum_{i=1}^{\compgamma}m_i\Gamma_i$ as the original sequence $\{\Sigma^j\}_{j\in\N}$. 
Furthermore, within any ball that is concentric with $B$ but has smaller radius, the convergence (up to a subsequence) is smooth away from finitely many points in the singular locus of the action of $G$. 
\item\label{lpc:conncomp} 
If $\hat\Sigma^j\subset\Sigma^j$ is the union of some of the connected components of $\Sigma^j\cap B$, then a subsequence of $\{\Phi_k^j(1,\hat\Sigma^j)\}_{k\in\N}$ converges (with multiplicity one) in the sense of varifolds in $B$ to $\hat V^j\subset B$, which is the union of some of the connected components of $V^j$ and satisfies $\partial\hat V^j\setminus\partial M=\partial\hat\Sigma^j\setminus\partial M$. 
Moreover, if $\hat V^j$ intersects $\partial M$, then $\hat\Sigma^j$ also has a nonempty intersection with $\partial M$.
\item\label{lpc:conncompconv} 
Let $\hat\Sigma^j\subset\Sigma^j$ and $\hat V^j$ be as in statement \ref{lpc:conncomp}.
Then a subsequence of $\{\hat V^j\}_{j\in\N}$ converges in the sense of varifolds in $B$ to $\sum_{i=1}^\compgamma \hat m_i\Gamma_i$, where $\hat m_i\in\{0,\ldots,m_i\}$ for $i=1,\ldots,\compgamma$. 
Moreover, within any ball that is concentric with $B$ but has smaller radius, the convergence is smooth away from finitely many points in the singular locus of the action of $G$.
\end{enumerate}
\end{lemma}
 
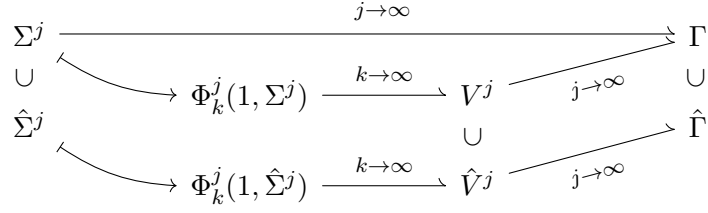
\begin{figure}%
\centering
\pgfmathsetmacro{\winkel}{15}
\begin{tikzpicture}[line cap=round,line join=round,every node/.style={anchor=base,inner sep=5pt}]
\path(0,1.2)node(sigma){$\Sigma^j$} 
++(-\winkel:3)node(phi){$\Phi_k^j(1,\Sigma^j)$} 
++(-0:3)node(vj){$V^j$} 
++(\winkel:3)node(gamma){$\Gamma$}
;
\path(0,0)node(sigmahat){$\hat\Sigma^j$} 
++(-\winkel:3)node(phihat){$\Phi_k^j(1,\hat\Sigma^j)$} 
++(-0:3)node(vjhat){$\hat V^j$} 
++(\winkel:3)node(gammahat){$\hat\Gamma$}
; 
\path(sigma)-|(gamma.west)coordinate[midway](gammaW)
(phi)-|(vj.west)coordinate[midway](vjW)
(phihat)-|(vjhat.west)coordinate[midway](vjhatW); 
\draw[arrows={->[harpoon,length=1mm]}](sigma)--(gammaW)  node[pos=0.525,above]{$\scriptstyle j\to\infty$};
\draw[arrows={->[harpoon,length=1mm]}](vj)   --(gamma)   node[midway,below,sloped]{$\scriptstyle j\to\infty$};
\draw[arrows={->[harpoon,length=1mm]}](vjhat)--(gammahat)node[midway,below,sloped]{$\scriptstyle j\to\infty$};
\draw[arrows={->[harpoon,length=1mm]}](phi)   --(vjW)    node[midway,above]{$\scriptstyle k\to\infty$};
\draw[arrows={->[harpoon,length=1mm]}](phihat)--(vjhatW) node[midway,above]{$\scriptstyle k\to\infty$};
\draw[|->](sigma)to[bend right=\winkel](phi.west);
\draw[|->](sigmahat)to[bend right=\winkel](phihat.west);
\path(sigmahat)--(sigma)node[midway,sloped,above=-.6em]{$\subset$};
\path(vjhat)--(vj)node[midway,sloped,above=-.6em]{$\subset$};
\path(gammahat)--(gamma)node[midway,sloped,above=-.7em]{$\subset$};
\end{tikzpicture}
\caption{Diagram of the varifold convergence in a small ball $B$ described in Lemma~\ref{lem:LocalPictureConv}.}
\label{fig:scheme_local}
\end{figure}

\begin{remark}
In statement \ref{lpc:conncompconv} we do not make any claims about the convergence of $\{\hat\Sigma^j\}_{j\in\N}$. 
It is unclear whether the sequence $\{\hat\Sigma^j\}_{j\in\N}$ is still $G$-almost minimizing in $B$, hence we cannot apply the same arguments as for $\{\Sigma^j\}_{j\in\N}$.

By smooth convergence away from a finite set $\mathcal{P}$ we mean that around every point in $M\setminus\mathcal{P}$ there exists a ball where (with respect to some suitable coordinate chart) the convergence is smooth and graphical, possibly with multiplicity if specified (as e.\,g.~in \ref{lpc:conncompconv}).
\end{remark}

\begin{proof}
\begin{description}[wide]
\item[\ref{lpc:vj}] 
Standard compactness arguments yield that, given any $j\in\N$, a subsequence of $\{\Phi^{j}_{k}(1,\Sigma^j)\}_{k\in\N}$ converges in the sense of varifolds to some $V^j$. 
In $B$, the limit $V^j$ is induced by a smooth, $G$-stable minimal surface with multiplicity one
which satisfies the desired boundary conditions by \cite[Theorem~13.4.3]{Franz2022}. 
Moreover, by the same theorem, its genus and area are bounded uniformly with respect to $j$. 
Note that $V^j$ is properly embedded in $B$ because it satisfies the boundary condition $\partial V^j\setminus\partial M=\Sigma^j\cap\partial B\setminus\partial M$ and cannot have interior touching points on $\partial M$ because $\partial M$ is assumed to be strictly mean convex. 
 
\item[\ref{lpc:vjconv}] We may choose a sequence of indices $k_j\in\N$ such that the distance between $\Phi^j_{k_j}(1,\Sigma^j)$ and $V^j$ is less than $1/j$ (with respect to a metric inducing the varifold convergence, 
see \cite[pp.\,66]{Pitts1981} or \cite[pp.\,703]{MarquesNeves2014}). 
Since the initial sequence $\{\Sigma^j\}_{j\in\N}$ converges to $\Gamma=\sum_{i=1}^\compgamma m_i\Gamma_i$ in the sense of varifolds, 
\cite[Lemma~3.7]{DeLellisPellandini2010} implies that $\{\Phi^j_{k_j}(1,\Sigma^j)\}_{j\in\N}$, and thus $\{V^j\}_{j\in\N}$, also converge in the sense of varifolds to $\Gamma$. 

As observed in the proof of part \ref{lpc:vj}, genus and area of $V^j$ are bounded uniformly with respect to $j$.  
Ilmanen's localized Gauss-Bonnet estimate (see \cite[Lecture~3]{Ilmanen1998}) then implies uniform integral bounds on the squared norm of the second fundamental form of $V^j$ in the ball $B'\subset B$ with the same center as $B$ but only half the radius. 
Consequently, a subsequence of $\{V^j\}_{j\in\N}$ converges smoothly away from finitely many points where curvature may concentrate. 
These points lie on the singular locus of the action of $G$, because $G$-stability of $V^j$ implies curvature estimates in $B'$ away from the singular locus.  

\item[\ref{lpc:conncomp}]
Following the proof of \cite[Lemma~9.1]{DeLellisPellandini2010}, 
the idea is that the sequence $\{\Phi^{j}_k(1,\Sigma^j)\}_{k\in\N}$ converges in the sense of varifolds with multiplicity one, and this has further implications. 
In particular, the convergence is consistent with the convergence in the sense of currents, wherein  the behaviour of the boundary is well-defined.

By standard varifold and currents compactness, a subsequence of $\{\Phi^{j}_k(1,\hat\Sigma^j)\}_{k\in\N}$ converges in $B$ to a varifold $\hat V^j=V$ in the sense of varifolds and to an integer rectifiable current $T$ in the sense of currents. 
Observe that the boundary $\partial T$ of $T$ is the limit in the sense of currents as $k\to\infty$ of the current induced by $\partial \Phi^j_k(1,\hat\Sigma^j)$ (with positive orientation). 
Thus, $\partial T\setminus \partial M$ coincides with the current induced by $\partial\hat\Sigma^j\setminus \partial M$ 
because $\partial \Phi^j_k(1,\hat\Sigma^j)\setminus \partial M=\partial \hat\Sigma^j\setminus \partial M$ for all $k$.  

Denoting by $\nm{V}$ and $\nm{T}$ the measures on $B$ induced by $V$ and $T$, we have that
\begin{equation}\label{eq:measineq}
\nm{T}\le \nm{V}\le \nm{V^j} = \sum_{i=1}^L\hsd^2\mres \Delta_i,
\end{equation}
where $\Delta_i$ are the connected components of the surface $V^j$. Since $\partial T$ and $\partial \Delta_i$ lie on $\partial B$, there exist integers $h_1,\ldots,h_L$ (as in \cite[Lemma~9.1]{DeLellisPellandini2010}) such that
\[
T=\sum_{i=1}^L h_i [[ \Delta_i ]],
\]
where $[[ \Delta_i ]]$ is the current induced by $\Delta_i$ (with positive orientation).
In particular, by \eqref{eq:measineq}, we have that $h_i\in \{-1,0,1\}$.
Note that each $\Delta_i$ has nontrivial intersection with $\partial B\setminus\partial M$, provided that $B$ has sufficiently small radius such that it does not contain any closed minimal surfaces. 
As a result, since $\partial T\setminus\partial M=[[\partial\hat\Sigma^j\setminus\partial M]]$, we have that $h_i\in\{0,1\}$ (here, we are using that the currents induced by both $\Delta_i$ and $\partial\hat\Sigma^j$ are chosen with positive orientation), and $h_i=1$ if and only if $\partial\Delta_i\setminus\partial M\subset \partial \hat\Sigma^j$.

Let $V'$ and $T'$ be the limits in the sense of varifolds and in the sense of currents, respectively, of $\{\Phi^{j}_k(1,\Sigma^j\setminus\hat\Sigma^j)\}_{k\in\N}$.
Arguing the same as before we have that $T'=\sum_{i=1}^L h_i'
[[\Delta_i]]$ for some $h_i'\in\{0,1\}$.
By construction, $\{\Phi^{j}_k(1,\Sigma^j)\}_{k\in\N}$ converges in the sense of currents to $T+T'=\sum_{i=1}^L(h_i+h_i')[[\Delta_i]]$. Moreover, since $\partial V^j\setminus \partial M =\Sigma^j\cap B\setminus\partial M$ by part \ref{lpc:vj}, we have 
\[
\sum_{i=1}^L[[\partial \Delta_i\setminus\partial M]]
=[[\Sigma^j\cap B \setminus\partial M]] 
=\partial(T+T')\setminus\partial M
=\sum_{i=1}^L(h_i+h_i')[[\partial\Delta_i\setminus\partial M]],
\]
which implies that $h_i+h_i'=1$ for all $i=1,\ldots,L$, and therefore $\nm{V}=\nm{T}$ and $\nm{V'}=\nm{T'}$. 
This proves that $\hat V^j=V$ is the varifold induced by the union of some of the connected components of $V^j$. 
Moreover, since $\partial T\setminus \partial M = [[\partial\hat\Sigma^j\setminus\partial M]]$ and $\nm{T}=\nm{V} = \nm{\hat V^j}$, we have that $\hat V^j\cap\partial M\not=\emptyset$ implies $\hat\Sigma^j\cap \partial M\not=\emptyset$.

\item[\ref{lpc:conncompconv}]
By \ref{lpc:conncomp}, $\hat V^j$ is the union of some of the connected components of $V^j$ and $\{V^j\}_{j\in\N}$ converges to $\Gamma$ as stated in \ref{lpc:vjconv}. 
On the one hand, by standard varifold compactness, a subsequence of $\{\hat V^j\}_{j\in\N}$ converges in the sense of varifolds in $B$ to a varifold $\hat\Gamma$. 
On the other hand, given any ball $B'$ that is concentric with $B$ but has smaller radius, statement \ref{lpc:vjconv} implies that a further subsequence of $\{\hat V^j\}_j$ converges smoothly in $B'$ away from finitely many points to $\sum_{i=1}^\compgamma\hat m_i\Gamma_i$ for some $\hat m_i\in\{0,\ldots, m_i\}$. 
In particular, $\hat\Gamma=\sum_{i=1}^\compgamma\hat m_i\Gamma_i$ in $B'$. 
Therefore, the multiplicities $\hat m_i$ do not depend on the choice of $B'$ and the claim follows. 
\qedhere
\end{description}
\end{proof}

Below, we prove the tree lifting lemma following the arguments in \cite[§\,4]{DeLellisPellandini2010}: the general idea is to lift the edges locally in small balls and to show that this can be done ``consistently'' in different balls.  
A new difficulty occurs along edges with leaves on $\partial M$. 
Claim 1 in our proof deals with the transition from an interior point to a point possibly on the boundary and is analogous to \cite[Lemma~4.1]{DeLellisPellandini2010} about the ``continuation of the leaves'' (not to be confused with the leaves of a tree).

\begin{proof}[Proof of Lemma~\ref{lem:BoundarySimonLemma}]
Let $r\colon M \to\interval{0,\infty}$ be the $G$-invariant function from Proposition \ref{prop:SmoothLimitAlmMin}. 
By compactness of $M$, there exists a finite set $C\subset M$ such that $\{B_{r(z)/2} (z)\st z\in C\}$ is a finite covering of $M$. 
Given any open ball $B\subset M\setminus C$ with radius $r<\min_{z\in C}r(z)/2$ 
there exists some $z\in C$ such that $B\subset B_{r(z)}(z)\setminus\{z\}$. 
The $G$-almost minimality in annuli in $\mathcal{AN}_{r(z)}(z)$ 
then implies that $\{\Sigma^j\}_{j\in\N}$ is $G$-almost minimizing in $B$. 

Let $T$ be a properly embedded, rooted tree in $\Gamma_i$, with vertices 
$v_0,\ldots,v_n\in\Gamma_i$, edges $\gamma_1,\ldots,\gamma_n\subset\Gamma_i$ and leaves 
$\{\ell_1,\ldots,\ell_q\}=\{v_1,\ldots,v_n\}\cap\partial\Gamma_i$.  
In particular, every edge has at most one leaf. 
Up to a slight perturbation, we may assume that the edges of $T$ do not contain any points in $C$. 
After perturbing the tree slightly near its leaves, we may choose $0<\rho\leq\varepsilon_0$ such that 
for every edge $\gamma$ containing a leaf $\ell$ the segment $\gamma\cap B_\rho(\ell)$ is a geodesic meeting $\partial M$ orthogonally, and 
\begin{align}
\label{eqn:rho}
\dist_M\biggl(
\Bigl(\bigcup_{k=1}^{n}\gamma_k\Bigr)	
\setminus \Bigl(\bigcup_{k=1}^{q}B_\rho(\ell_k)\Bigr),\partial M\biggr)\geq\rho, 
\end{align}
and such that for every edge $\gamma$ there exists a finite set of points $x_0,\ldots,x_N$ on $\gamma$ (labeled consecutively, where $N$ may depend on $\gamma$) with the following properties (cf Figure \ref{fig:treelifting}). 
\begin{enumerate}[label={\normalfont(\alph*)}]
\item
$x_0$ and $x_N$ are vertices, i.\,e. the endpoints of the edge $\gamma$.  
\item 
Denoting the geodesic segment in $M$ between $x_\alpha$ and $x_{\alpha+1}$ by $[x_\alpha,x_{\alpha+1}]$, there is a homotopy between the curves $\gamma$ and $\sum_{\alpha=0}^{N-1}[x_\alpha,x_{\alpha+1}]$ in $U_\varepsilon\Gamma_i$ fixing their endpoints.
\item 
The balls $B^{\alpha}\vcentcolon=B_\rho(x_\alpha)$ are pairwise disjoint and $B^{\alpha}\cap\partial M=\emptyset$ for all $\alpha\in\{0,\ldots,N-1\}$.
\item 
$B^\alpha\cup B^{\alpha+1}$ is contained in a ball $B^{\alpha,\alpha+1}$ of radius $3\rho$ 
which satisfies $B^{\alpha,\alpha+1}\cap C=\emptyset$ so that $\Sigma^j$ is $(G,\delta_j,\epsilon_j)$-almost minimizing in $B^{\alpha,\alpha+1}$ for some $\delta_j,\epsilon_j$ with $\epsilon_j\to0$.
\item\label{ab:transverse} $\partial B^\alpha$ intersects $\Sigma^j$ transversely for all $\alpha\in\{0,\ldots,N\}$ and all $j\in\N$.
\end{enumerate}

\begin{figure}%
\centering
\pgfmathsetmacro{\Number}{6}
\pgfmathsetmacro{\Numberminusone}{\Number-1}
\pgfmathsetmacro{\rhopar}{0.7}
\pgfmathsetmacro{\tA}{1-1/\Number}
\pgfmathsetmacro{\tB}{1-2/\Number}
\begin{tikzpicture}[line cap=round,line join=round,semithick] 
\begin{scope}
\clip (0,3*\rhopar)coordinate(A)to[bend left=10]coordinate[midway](xN)(0,-3*\rhopar)coordinate(B)-|(-3-3*\rhopar*\Number,0)|-cycle;
\draw[FarbeB](xN)circle(\rhopar);
\pgfresetboundingbox
\end{scope}
\draw[thin](A)to[bend left=10](B)node[above left]{$\partial M$};
\begin{scope}
\clip(xN)circle(\rhopar+0.01);
\draw[FarbeB](A)to[bend left=10](B);
\pgfresetboundingbox
\end{scope}
\path(A)to[bend left=10](B);
\draw[thick](xN)to[out=180,in=15]
node foreach[count=\i] \t in {\tA,\tB,...,-0.00001} [pos=\t](x\i){}
++({-3*\rhopar*\Number},0)coordinate(x0);
\draw[thick] (x0)to[bend right=10]node[pos=1,left,sloped]{$\cdots$}++(150:2.5)node[below left]{$T$};
\draw[thick] (x0)to[bend right=10]node[pos=1,left,sloped]{$\cdots$}++(230:1.5);
\foreach\i in {0,...,\Numberminusone}{
\draw plot[bullet](x\i);
\draw[FarbeB](x\i)circle(\rhopar);
}
\draw[FarbeB](xN)++(177:\rhopar)arc(177:183:\rhopar);
\draw(x0)node[anchor=120,circle,inner sep=1pt]{$x_0$}
plot[bullet](xN)node[anchor=45,circle,inner sep=1pt]{$x_N$}
(x3)node[below]{$x_\alpha$}++(0,\rhopar)node[below,FarbeB]{$B^{\alpha}$}
(x4)node[below]{$x_{\alpha+1}$}++(0,\rhopar)node[below,FarbeB]{$B^{\alpha+1}$}
;
\draw[FarbeB,latex-latex](x1.center)--++(-60:\rhopar)node[below=-1pt,midway,sloped]{$\scriptstyle\rho$};
\draw[FarbeB]($0.5*(x3)+0.5*(x4)+(0,-0.37)$)coordinate(c)circle(3*\rhopar)++(0,-3*\rhopar)node[above]{$B^{\alpha,\alpha+1}$};
\draw[FarbeB,latex-latex](c)--++(-60:3*\rhopar)node[midway,below=-1pt,sloped]{$\scriptstyle3\rho$};
\end{tikzpicture}
\caption{Small balls $B^{\alpha}\vcentcolon=B_\rho(x_\alpha)\subset M$ around points $x_0,\ldots,x_N$ on the edge $\gamma$ of the tree $T$. }%
\label{fig:treelifting}%
\end{figure}
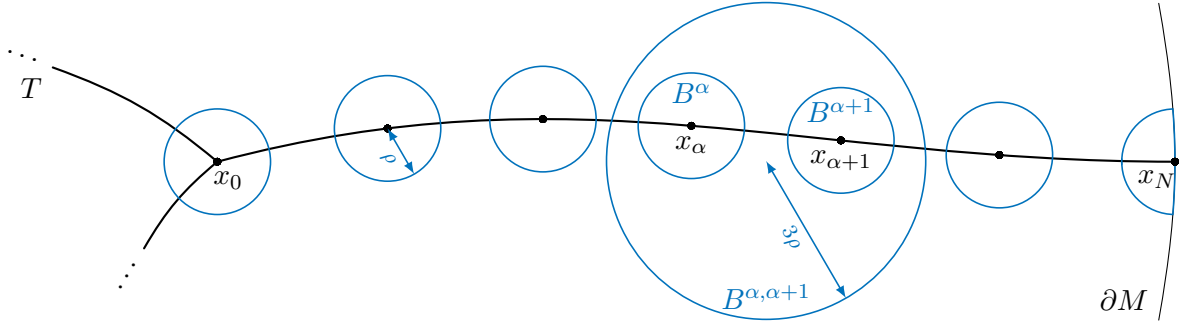 

We choose an interior vertex $v=x_0$ as root which has an edge $\gamma$ containing a leaf $\ell=x_N$ and focus our analysis on this edge for now. 
Given $j\in\N$ and $\alpha\in\{0,\ldots,N\}$, let $\{\Phi_k^{j,\alpha}\}_{k\in\N}\subset\Is_G^{\delta_j}(B^\alpha,\Sigma^j)$ be a minimizing sequence of isotopies in the sense of equation \eqref{eq:DefPhi}. 
In particular, $\Phi_k^{j,\alpha}$ coincides with the identity on $M\setminus B^{\alpha}$. 
In what follows, we assume $\alpha\leq N-1$, so that $B^{\alpha}$ is in the interior of $M$. 
Since $B^\alpha\cap B^{\alpha+1}=\emptyset$, the map $\Phi_k^{j,\alpha,\alpha+1}\colon[0,1]\times M\to M$ given by 
\begin{align}\label{eqn:combined_isotopy}
\Phi_k^{j,\alpha,\alpha+1}(t,x)&=\left\{\begin{aligned}
\Phi_k^{j,\alpha}\bigl(\min\{2t,1\},x\bigr) &&& \text{ if $x\in B^\alpha$,} \\
\Phi_k^{j,\alpha+1}\bigl(\max\{2t-1,0\},x\bigr) &&& \text{ if $x\in B^{\alpha+1}$,} \\
x\hphantom{)} &&& \text{ otherwise}
\end{aligned}	\right.
\end{align}
is in $\Is_G^{\delta_j}(B^{\alpha,\alpha+1}, \Sigma^j)$. 
By Lemma~\ref{lem:LocalPictureConv}~\ref{lpc:vj} applied in $B^{\alpha}$, a subsequence of $\{\Phi_k^{j,\alpha,\alpha+1}(1,\Sigma^j)\}_{k\in\N}$ converges in the sense of varifolds in $B^{\alpha}$ to a smooth, $G$-stable minimal surface $V^{j,\alpha}$ in $B^{\alpha}$. 
By the same argument, a further subsequence converges in the sense of varifolds in $B^{\alpha+1}$ to a smooth, $G$-stable minimal surface $V^{j,\alpha+1}$ in $B^{\alpha+1}$ 
which satisfies the free boundary condition on $B^{\alpha+1}\cap \partial M$ if the latter is nonempty, i.\,e. in the case $\alpha=N-1$. 
By standard varifold compactness, an even further subsequence of $\{\Phi_k^{j,\alpha,\alpha+1}(1,\Sigma^j)\}_{k}$ converges in the sense of varifolds in $B^{\alpha,\alpha+1}$ to some $V^{j,\alpha,\alpha+1}$ which coincides with $V^{j,\alpha}$ in $B^{\alpha}$ and with $V^{j,\alpha+1}$ in $B^{\alpha+1}$.

Given any connected component $\hat\Sigma^{j,\alpha}$ of $\Sigma^j\cap B^\alpha$, Lemma~\ref{lem:LocalPictureConv}~\ref{lpc:conncomp} implies that a further subsequence of $\{\Phi_k^{j,\alpha,\alpha+1}(1,\hat\Sigma^{j,\alpha})\}_{k}$ converges in the sense of varifolds to the union $\hat V^{j,\alpha}$ of some of the connected components of $V^{j,\alpha}$.
By Lemma~\ref{lem:LocalPictureConv}~\ref{lpc:vjconv}, $\{V^{j,\alpha}\}_{j\in\N}$ converges in the sense of varifolds in $B^\alpha$ to $m_i\Gamma_i$ and in $B_{\rho/2}(x_\alpha)$ the convergence in smooth away from finitely many points in the singular locus of the group action. 
In particular, the surface $\hat V^{j,\alpha}\cap B_{\rho/2}(x_\alpha)$ is arbitrarily close, smoothly away from finitely many points, to $\hat m^{j,\alpha}\in\{0,\ldots,m_i\}$ copies of $\Gamma_i\cap B_{\rho/2}(x_\alpha)$ provided that $j$ is sufficiently large.

\emph{Claim 0.} In the notation of the previous paragraph, it is possible to choose the connected component $\hat\Sigma^{j,\alpha}$ such that $\hat m^{j,\alpha}\geq1$ provided that $j$ is sufficiently large.

\begin{proof}
Since $V^{j,\alpha}\cap B_{\rho/2}(x_\alpha)$ is arbitrarily close (for $j$ large) to $m_i\Gamma_i\cap B_{\rho/2}(x_\alpha)$ smoothly away from a finite set $\mathcal{P}$, each of its connected components is either arbitrarily close to a positive number of copies of $\Gamma_i\cap B_{\rho/2}(x_\alpha)$ (smoothly away from $\mathcal{P}$) or it is contained in a small neighbourhood of $\mathcal{P}$.
Hence, given $\hat\Sigma^{j,\alpha}$ and the corresponding $\hat V^{j,\alpha}\subset V^{j,\alpha}$ as above, we have that $\hat V^{j,\alpha}\cap B_{\rho/2}(x_\alpha)$ is either close to $\hat m^{j,\alpha}\geq 1$ copies of $\Gamma_i\cap B_{\rho/2}(x_\alpha)$ (smoothly away from $\mathcal{P}$) or it is contained in a small neighbourhood of $\mathcal{P}$. 

Assume by contradiction that the latter case holds for every choice of $\hat\Sigma^{j,\alpha}$. 
Then $V^{j,\alpha}\cap B_{\rho/2}(x_\alpha)$ is also contained in a small neighbourhood of $\mathcal{P}$, being the union of $\hat V^{j,\alpha}\cap B_{\rho/2}(x_\alpha)$ over all the possible choices of $\hat\Sigma^{j,\alpha}$.
This contradicts the fact that $V^{j,\alpha}\cap B_{\rho/2}(x_\alpha)$ is close to $m_i\Gamma_i\cap B_{\rho/2}(x_\alpha)$, where $m_i\geq1$ and the claim is proved. 
\end{proof}

Let $\hat\Sigma^{j,\alpha}$ be as in Claim 0 and let $\hat\Sigma^{j,\alpha,\alpha+1}$ be the connected component of $\Sigma^j\cap B^{\alpha,\alpha+1}$ containing $\hat\Sigma^{j,\alpha}$.  
By Lemma~\ref{lem:LocalPictureConv}~\ref{lpc:conncomp}, a further subsequence of $\{\Phi_k^{j,\alpha,\alpha+1}(1,\hat\Sigma^{j,\alpha,\alpha+1})\}_{k}$ converges in the sense of varifolds in $B^{\alpha+1}$ to the union $\hat V^{j,\alpha+1}$ of some of the connected components of $V^{j,\alpha+1}$ and an even further subsequence has a varifold limit $\hat V^{j,\alpha,\alpha+1}$ in $B^{\alpha,\alpha+1}$ which coincides with $\hat V^{j,\alpha}$ in $B^{\alpha}$ and with $\hat V^{j,\alpha+1}$ in $B^{\alpha+1}$.

\emph{Claim 1.} In $B_{\rho/2}(x_{\alpha+1})$ the surface $\hat V^{j,\alpha+1}$ is arbitrarily close (smoothly away from finitely many points on the singular locus of the group action) to a positive number of copies of $\Gamma_i\cap B_{\rho/2}(x_{\alpha+1})$ provided that $j$ is sufficiently large.

\begin{proof}
By Lemma~\ref{lem:LocalPictureConv}~\ref{lpc:conncompconv}, a subsequence of $\{\hat V^{j,\alpha+1}\}_{j\in\N}$ converges in the sense of varifolds in $B^{\alpha+1}$ to an integer multiple $\hat m^{\alpha+1}\Gamma_i$ of $\Gamma_i$, and in $B_{\rho/2}(x_{\alpha+1})$, away from a finite set $\mathcal{P}$ of concentration points on the singular locus, the convergence is smooth. 
We want to prove that $\hat m^{\alpha+1}\ge 1$.
Towards a contradiction, assume that $\hat m^{\alpha+1}=0$. 
Then, $\hat V^{j,\alpha+1}\cap B_{\rho/2}(x_{\alpha+1})$ is contained in an arbitrarily small neighbourhood of $\mathcal{P}$ provided that $j$ is sufficiently large. 
Up to a further subsequence we may extract a varifold limit $W$ of $\{\hat V^{j,\alpha,\alpha+1}\}_{j}$ introduced in the previous paragraph.  
On the one hand, 
\begin{align}\label{eqn:W(alpha+1)}
\nm{W}\bigl(B_{\rho/2}(x_{\alpha+1})\bigr)&=0,  
\intertext{because $\hat V^{j,\alpha,\alpha+1}$ coincides with $\hat V^{j,\alpha+1}$ in $B^{\alpha+1}$.
On the other hand, Claim 0 implies }
\label{eqn:W(alpha)}
\nm{W}\bigl(B_{\rho/2}(x_\alpha)\bigr)&\geq\hsd^2\bigl(\Gamma_i\cap B_{\rho/2}(x_\alpha)\bigr). 
\end{align}
For all $j$ in the subsequence, let $k_j\in\N$ be sufficiently large such that 
\[
\bigl\{\Phi_{k_j}^{j,\alpha,\alpha+1}(1,\hat \Sigma^{j,\alpha,\alpha+1})\bigr\}_{j}
\]
 also converges to $W$.
By \cite[Lemma~3.7]{DeLellisPellandini2010}, $\{\Phi_{k_j}^{j,\alpha,\alpha+1}(1,\Sigma^j)\}_{j\in\N}$ converges in the sense of varifolds to $m_i\Gamma_i$ in $B^{\alpha,\alpha+1}$. 
Hence, $W\le m_i\Gamma_i\cap B^{\alpha,\alpha+1}$ as varifolds. 
Following \cite[(4.1--5)]{DeLellisPellandini2010}, we obtain $\nm{W}(\partial B_\tau(z)) = 0$ for any ball $B_\tau(z)\subset B^{\alpha,\alpha+1}$, which implies that the function
\(
z\mapsto \nm{W}\bigl(B_{\rho/2}(z)\bigr)
\)
is continuous for $z$ varying in the geodesic segment $[x_\alpha,x_{\alpha+1}]$.  
In view of \eqref{eqn:W(alpha+1)} and \eqref{eqn:W(alpha)}, there exists some $z_\alpha\in[x_\alpha,x_{\alpha+1}]$ such that 
\[
\nm{W}\bigl(B_{\rho/2}(z_\alpha)\bigr)=\frac{1}{2}\hsd^2\bigl(\Gamma_i\cap B_{\rho/2}(z_\alpha)\bigr).
\]
Moreover, $B_{\rho/2}(z_\alpha)\cap\partial M=\emptyset$ because $\nm{W}\bigl(B_{\rho/2}(z)\bigr)=0$ for all $z\in[x_\alpha,x_{\alpha+1}]$ at distance less than $\rho/2$ from $\partial M$, since $W$ vanishes in $B^{\alpha+1}$. 
Here, we also rely on property \eqref{eqn:rho} and on the fact that $\gamma\cap B_\rho(\ell)$ is a geodesic meeting $\partial M$ orthogonally. 
In particular, 
\begin{align}\label{DP:46}
\lim_{j\to\infty} \hsd^2\Bigl(\Phi_{k_j}^{j,\alpha,\alpha+1}(1,\hat\Sigma^{j,\alpha,\alpha+1})\cap B_{\rho/2}(z_\alpha)\Bigr)
&=\frac{1}{2}\hsd^2\bigl(\Gamma_i\cap B_{\rho/2}(z_\alpha)\bigr),
\\
\label{DP:47}
\lim_{j\to\infty} \hsd^2\Bigl(\Phi_{k_j}^{j,\alpha,\alpha+1}(1,\Sigma^j\setminus \hat\Sigma^{j,\alpha,\alpha+1})\cap B_{\rho/2}(z_\alpha)\Bigr) 
&= \Bigl(m_i-\frac 12\Bigr)\hsd^2(\Gamma_i\cap B_{\rho/2}(z_\alpha)).
\end{align}
We many now conclude as in \cite{DeLellisPellandini2010} to reach a contradiction. 
Indeed, equations \eqref{DP:46} and \eqref{DP:47} correspond exactly to \cite[(4.6--7)]{DeLellisPellandini2010} and $B_{\rho/2}(z_\alpha)$ is contained in the interior of $M$. 
The idea is that (after performing a further minimization process as in Lemma~\ref{lem:LocalPictureConv}) the two surfaces $\Phi_{k_j}^{j,\alpha,\alpha+1}(1,\hat\Sigma^{j,\alpha,\alpha+1})\cap B_{\rho/2}(z_\alpha)$ and $\Phi_{k_j}^{j,\alpha,\alpha+1}(1,\Sigma^j\setminus \hat\Sigma^{j,\alpha,\alpha+1})\cap B_{\rho/2}(z_\alpha)$ become close to an integer multiple of $\Gamma_i$ and this contradicts equations \eqref{DP:46} and \eqref{DP:47}. 
The technical details are exactly the same as in \cite{DeLellisPellandini2010} and therefore not repeated here. 
\end{proof}

\emph{Claim 2.} In the case $\alpha=N-1$, the intersection $\hat\Sigma^{j,N}\cap\partial M$ is nonempty provided that $j$ is sufficiently large.

\begin{proof}
The free boundary of $\Gamma_i$ intersects $B_{\rho/2}(x_N)$ because by assumption $x_N$ is a leaf of the properly embedded tree in question. 
In fact, since $\rho$ is small, we may assume that $\partial\Gamma_i\cap B_{\rho/2}(x_N)$ consists of a single boundary segment connecting two points on $\partial B_{\rho/2}(x_N)\cap\partial M$. 
By Claim~1, $\hat V^{j,N}\cap B_{\rho/2}(x_N)$ is arbitrarily close (smoothly away from finitely many points) to a positive number of copies of $\Gamma_i\cap B_{\rho/2}(x_N)$ provided that $j$ is sufficiently large along a subsequence. 
In particular, some boundary segment of $\hat V^{j,N}\cap B_{\rho/2}(x_N)$ is close to the boundary of $\Gamma_i\cap B_{\rho/2}(x_N)$.
Thus, being properly embedded, $\hat V^{j,N}$ intersects $\partial M$. 
Consequently, $\hat\Sigma^{j,N}$ intersects $\partial M$ as well by Lemma~\ref{lem:LocalPictureConv}~\ref{lpc:conncomp}, and the claim follows.
\end{proof}

Given the previous three claims, we now lift the edge $\gamma$, connecting an interior point $v=x_0$ to a leaf $\ell=x_N$ as follows.
We apply Claim 0 for $\alpha=0$ and obtain (for all sufficiently large $j$ in a subsequence) a connected component $\hat\Sigma^{j,0}$ of $\Sigma^j\cap B^0$ such that a subsequence of $\{\Phi_k^{j,0}(1,\hat\Sigma^{j,0})\}_{k\in\N}$ has a nonzero varifold limit in $B_{\rho/2}(x_0)$. We fix a point $y_0$ in $\hat\Sigma^{j,0}\subset\Sigma^j$.
Then we apply Claim~1 iteratively at every subsequent point $x_{\alpha+1}$ with $\alpha=0,\ldots,N-1$ along $\gamma$.
In every step, we obtain a connected component $\hat\Sigma^{j,\alpha+1}$ belonging to the same connected component of $\Sigma^j\cap B^{\alpha,\alpha+1}$ as the previously selected $\hat\Sigma^{j,\alpha}$ such that a subsequence of $\{\Phi_k^{j,\alpha+1}(1,\hat\Sigma^{j,\alpha+1})\}_{k\in\N}$ again has a nonzero varifold limit in $B_{\rho/2}(x_{\alpha+1})$. 
In particular, we may choose a point $y_{\alpha+1}\in\hat\Sigma^{j,\alpha+1}$ and a curve $\gamma^j_{\alpha+1}\subset\Sigma^j\cap B^{\alpha,\alpha+1}$ connecting the previously selected $y_\alpha$ with $y_{\alpha+1}$. 
Note that, thanks to Claim~2, we can choose $y_N\in\hat\Sigma^{j,N}\cap \partial M$.

Hence, defining $\gamma^j$ as the concatenation of the curves $\gamma^j_{0},\ldots,\gamma^j_{N}$, it is straightforward to check that $\gamma^j$ is the desired lift of the edge $\gamma$ (cf.~\cite[pp.\,64]{DeLellisPellandini2010}).
We now repeat the same process iteratively on all the edges of the tree (e.\,g. using breadth-first search) except that at each new starting vertex, the connected component of $\Sigma^j$ and the first point $y_0$ have already been selected when lifting the preceding edge. 
Thereby, the lift of every new edge connects to the previously lifted part of the tree. 
This yields a properly embedded tree in $\Sigma^j$ which is properly homotopic to the initial tree in $U_{\varepsilon}\Gamma_i$, as desired. 
\end{proof}
 
\subsection{Proof of the main results}
 
The proof of our main theorems is based on the following result, for which we recall the notions of genus complexity $\gsum$ and boundary complexity $\bsum$ from Definition \ref{defn:abc} as well as the notation 
$\beta_1$ for the first Betti number (cf.~Appendix \ref{app:homology}).  

\begin{theorem}\label{thm:TopologyOfAlmMinSeq}
Let us assume to be in the setting of Proposition~\ref{prop:SmoothLimitAlmMin}. 
Then, for every sufficiently small $\varepsilon>0$, there exists $J_\varepsilon\in\N$ such that for all $j\geq J_\varepsilon$ there is a $G$-equivariant surface 
\(
\tilde\Sigma^j\subset U_{2\varepsilon}\Gamma 
\)
obtained from $\Sigma^j$ through surgery such that the sequence $\{\tilde\Sigma^j\}_{j\geq J_\varepsilon}$ also converges in the sense of varifolds to $\Gamma$ and we have the following bounds on the topology:
\begin{align*}
\beta_1(\Gamma)&\leq\liminf_{j\to\infty}\beta_1(\tilde\Sigma^j), &
\gsum(\Gamma)&\leq\liminf_{j\to\infty}\gsum(\tilde\Sigma^j), &
\bsum(\Gamma)&\leq\liminf_{j\to\infty}\bsum(\tilde\Sigma^j).  
\end{align*}
\end{theorem}

\begin{proof}
For the purpose of this proof, we do not distinguish between $\Gamma$ and its support, 
i.\,e.~we may regard $\Gamma$ as a smooth, compact, properly embedded, possibly disconnected surface, without taking multiplicity into account. 

Given any $\varepsilon>0$, we apply Proposition \ref{prop:PerformingSurgery} to $\{\Sigma^j\}_{j\in\N}$ obtaining $\tilde\Sigma^j$ from $\Sigma^j$ for all $j\geq J_\varepsilon$ through (equivariant) surgery such that $\tilde\Sigma^j\subset U_{2\varepsilon}\Gamma$ and $\tilde\Sigma^j\cap U_{\varepsilon}\Gamma=\Sigma^j\cap U_{\varepsilon}\Gamma$.
The new sequence $\{\tilde\Sigma^j\}_{j\geq J_\varepsilon}$ also converges to $\Gamma$ in the sense of varifolds. 
Given any $x\in M$, let $\mathrm{An}\in\mathcal{AN}_{r(x)}(x)$. 
If $\mathrm{An}\subset U_{\varepsilon}\Gamma$ then $\{\tilde\Sigma^j\}_{j\geq J_\varepsilon}$ is clearly $G$-almost minimizing in $\mathrm{An}$ because $\tilde\Sigma^j$ coincides with $\Sigma^j$ in $U_{\varepsilon}\Gamma$. 
In $M\setminus U_{\varepsilon}\Gamma$ the sequence $\{\tilde\Sigma^j\}_{j\geq J_\varepsilon}$ converges to zero in the sense of varifolds; hence we can achieve that $\{\tilde\Sigma^j\}_{j\geq J_\varepsilon}$ is $G$-almost minimizing in every $\mathrm{An}\in\mathcal{AN}_{r(x)}(x)$, by possibly reducing $r(x)>0$ depending on $\varepsilon$. 

By \cite[Theorem~2A.1]{Hatcher2002}, every element of the first homology group can be represented by a closed curve. 
Since the first Betti number is the rank of the first homology group, 
there exist $n\vcentcolon=\beta_1(\Gamma)$ closed curves $\gamma_1,\ldots,\gamma_n$ in $\Gamma$ which are linearly independent in the $\Z$-module $H_1(\Gamma)$. 
Assuming $2\varepsilon\leq\varepsilon_0$, Simon's lifting lemma, Lemma~\ref{lem:SimonLemma}, yields closed curves $\gamma_1^j,\ldots,\gamma_n^j\subset\tilde\Sigma^j\cap U_{2\varepsilon}\Gamma$ for any sufficiently large $j\in\N$, such that $\gamma_\ell^j$ is homotopic to a multiple of $\gamma_\ell$ in $U_{2\varepsilon}\Gamma$ for all $\ell\in\{1,\ldots,n\}$.  
In particular, $\gamma_1^j,\ldots,\gamma_n^j$ are linearly independent in the $\Z$-module $H_1(U_{2\varepsilon}\Gamma)\cong H_1(\Gamma)$. 
Suppose that $\alpha^j\vcentcolon=k_1\gamma_1^j+\ldots+k_{n}\gamma_n^j$ vanishes in the $\Z$-module $H_1(\tilde\Sigma^j)$ for some $k_1,\ldots, k_{n}\in\Z$. 
Then $\alpha^j$ also vanishes in the $\Z$-module $H_1(U_{2\varepsilon}\Gamma)$ since $\tilde\Sigma^j\subset U_{2\varepsilon}\Gamma$ and we obtain $k_1=\ldots=k_{n}=0$.  
Consequently, $H_1(\tilde\Sigma^j)$ has at least rank $n$, which concludes the proof of the inequality for the first Betti number $\beta_1$.

By Corollary~\ref{cor:homology2gsum}, the $\Z$-module $H_1(\Gamma)/\iota(H_1(\partial\Gamma))$ has rank $m\vcentcolon=2\gsum(\Gamma)$. 
Hence there exist closed curves $\gamma_1,\ldots,\gamma_m$ in $\Gamma$ which are linearly independent in the $\Z$-module $H_1(\Gamma)/\iota(H_1(\partial\Gamma))$ and as before, Simon's lifting lemma yields closed curves $\gamma_1^j,\ldots,\gamma_m^j\subset\tilde\Sigma^j\cap U_{2\varepsilon}\Gamma$ for any sufficiently large $j\in\N$, such that $\gamma_\ell^j$ is homotopic to a multiple of $\gamma_\ell$ in $U_{2\varepsilon}\Gamma$ for all $\ell\in\{1,\ldots,m\}$.  
Suppose that there exist $k_1,\ldots,k_{m}\in\Z$ such that 
$\alpha^j\vcentcolon=k_1\gamma_1^j+\ldots+k_{m}\gamma_m^j$ vanishes in the $\Z$-module $H_1(\tilde\Sigma^j)/\iota(H_1(\partial\tilde\Sigma^j))$. 
Then $\alpha^j$ can be represented by an integer linear combination of some boundary components of $\tilde\Sigma^j$. 
Since $\partial\tilde\Sigma^j\subset U_{2\varepsilon}\Gamma\cap\partial M$ and since  $U_{2\varepsilon}\Gamma\cap\partial M$ is a union of pairwise disjoint annuli, every boundary component of $\tilde\Sigma^j$ is homotopic to an integer multiple of some boundary component of $\Gamma$. 
Therefore, $k_1\gamma_1+\ldots+k_{m}\gamma_m$ vanishes in $H_1(\Gamma)/\iota(H_1(\partial\Gamma))$ which implies $k_1=\ldots=k_m=0$. 
Consequently, $H_1(\tilde\Sigma^j)/\iota(H_1(\partial\tilde\Sigma^j))$ has at least rank $m$, which (again by Corollary~\ref{cor:homology2gsum}) concludes the proof of the inequality for the genus complexity $\gsum$.   

It remains to prove the estimate for the boundary complexity $\bsum$. 
By assumption, $\varepsilon>0$ is sufficiently small such that there is a smooth retraction of $U_{2\varepsilon}\Gamma$ onto the support of $\Gamma=\sum_{i=1}^{\compgamma}m_i\Gamma_i$, 
where we recall that $\Gamma_1,\ldots,\Gamma_\compgamma$ are induced by pairwise disjoint, connected free boundary minimal surfaces in $M$. 
In particular, the sets $U_{2\varepsilon}\Gamma_i$ and $U_{2\varepsilon}\Gamma_l$ are disjoint for $i\neq l$. 
Since $\tilde\Sigma^j\subset U_{2\varepsilon}\Gamma$, it is clear that the surfaces $\tilde\Sigma^j\cap U_{2\varepsilon}\Gamma_i$ converge for $j\to\infty$ to $m_i\Gamma_i$ for every $i\in\{1,\ldots,\compgamma\}$ 
(see \cite[Lemma~3.1]{Li2015} for a reference that restriction to $U_{2\varepsilon}(\Gamma_i)$ respects varifold convergence in this case). 
By Remark \ref{rem:complexities}, $\bsum$ is additive with respect to taking unions of connected components. 
Therefore, we may assume without loss of generality that $\Gamma$ is connected. 
Otherwise, we prove the inequality $\bsum(\Gamma_i)\leq\bsum(\tilde\Sigma^j\cap U_{2\varepsilon}\Gamma_i)$ for each connected component $\Gamma_i$ separately and conclude by summation. 
(Note that even with the assumption that $\Gamma$ is connected, it could happen that the surface $\tilde\Sigma^j$ is disconnected.)

Let $b$ denote the number of boundary components of $\Gamma$. 
For each $k\in\{1,\ldots,b\}$, let $v_k\in\partial\Gamma$ be a point on the $k$th boundary component and let $T$ be a properly embedded, rooted tree in $\Gamma$ whose leaves are exactly given by the set $\{v_1,\ldots,v_b\}$. 
Such a tree can be constructed e.\,g. by connecting each $v_k$ to some interior root $v_0\in\Gamma$ via a smooth embedded curve $\gamma_k\subset\Gamma$.  
If $j\in\N$ is sufficiently large, then the tree lifting lemma, Lemma~\ref{lem:BoundarySimonLemma}, implies that there exists a properly embedded tree $T^j$ in $\tilde\Sigma^j$ which is properly homotopic to $T$ in $U_{2\varepsilon}\Gamma$.
Let $\hat\Sigma^j$ denote the connected component of $\tilde\Sigma^j$ that contains $T^j$.
By construction, $\hat\Sigma^j$ is contained in $U_{2\varepsilon}\Gamma$ and the $2\varepsilon$-tubular neighbourhoods of different boundary components of $\Gamma$ are disjoint. 
Since the leaves of all the trees in the homotopy are restricted to these $2\varepsilon$-tubular neighbourhoods, $\hat\Sigma^j$ must have at least $b$ boundary components for all sufficiently large $j$. 
Thus, recalling Definition~\ref{defn:abc}, 
\[
\bsum(\Gamma)=b-1
\leq\liminf_{j\to\infty}\beta_0(\partial\hat\Sigma^j)-1 
=\liminf_{j\to\infty}\bsum(\hat\Sigma^j)
\leq\liminf_{j\to\infty} \bsum(\tilde\Sigma^j). \qedhere
\]
\end{proof}

\begin{proof}[Proof of Theorems \ref{thm:main} and \ref{thm:mainorientable}]
Let $\{\Sigma^j\}_{j\in\N}$ be the min-max sequence from Theorem~\ref{thm:PreviousResults} and $\Gamma$ its varifold limit. 
By Lemma~\ref{lem:SeqAlmMin}, we may apply Theorem \ref{thm:TopologyOfAlmMinSeq} to obtain a sequence $\{\tilde\Sigma^j\}_{j\geq J_\varepsilon}$ such that $\tilde\Sigma^j$ is obtained from $\Sigma^j$ through surgery and such that 
\begin{align}\label{eqn:20230623-1}
\beta_1(\Gamma)&\leq\liminf_{j\to\infty}\beta_1(\tilde\Sigma^j), &
\gsum(\Gamma)&\leq\liminf_{j\to\infty}\gsum(\tilde\Sigma^j), & 
\bsum(\Gamma)&\leq\liminf_{j\to\infty}\bsum(\tilde\Sigma^j).  
\end{align} 
Moreover, for all $j\geq J_\varepsilon$, Lemmata \ref{lem:Betti_under_surgery} and \ref{lem:gsum_under_surgery} imply
\begin{align}
\label{eqn:20230623-2}
\beta_1(\tilde\Sigma^j)&\leq\beta_1(\Sigma^j), &
\gsum(\tilde\Sigma^j)&\leq\gsum(\Sigma^j).
\end{align}  
Combining \eqref{eqn:20230623-2} with the first and second estimate in \eqref{eqn:20230623-1} completes the proof of Theorem \ref{thm:main}.  

If $\Sigma^j$ is orientable for all $j\geq J_\varepsilon$ then Lemma \ref{lem:orientable_surgery} implies 
\[
\bsum(\tilde\Sigma^j)+\gsum(\tilde\Sigma^j)\leq\bsum(\Sigma^j)+\gsum(\Sigma^j),
\] 
which combined with the second and third estimate in \eqref{eqn:20230623-1} completes the proof of Theorem~\ref{thm:mainorientable}. 
\end{proof}

\section{Free boundary minimal surfaces in the unit ball with genus zero}\label{sec:genus0} 

In this section, we choose the Euclidean unit ball $\B^3\vcentcolon=\{(x,y,z)\in\R^3\st x^2+y^2+z^2\leq1\}$ as ambient manifold $M$. 
Given an oriented line $\ell\subset\R^3$ and an angle $\alpha\in\R$, let $\rotation_{\ell}^{\alpha}$
denote the rotation of angle $\alpha$ around $\ell$. 
Given $2\leq n\in\N$, let $\dih_n$ denote the subgroup of Euclidean isometries acting on $\B^3$ generated by the two rotations $\rotation_{x\text{-axis}}^{\pi}$ and $\rotation_{z\text{-axis}}^{2\pi/n}$. 
We call $\dih_n$ the \emph{dihedral group} of order $2n$. 
Clearly, $\dih_n$ is orientation-preserving, being generated by rotations, and it contains the cyclic group $\Z_n$ generated by $\rotation_{z\text{-axis}}^{2\pi/n}$ as a subgroup. 
The goal of this section is to provide a self-contained proof of the following existence result.

\begin{theorem}\label{thm:fbmsg0}
For each $2\leq n\in\N$ there exists an embedded, $\dih_{n}$-equivariant free boundary minimal surface $\fbmsgO_n$ in the Euclidean unit ball $\B^3$ with the following properties. 
\begin{itemize} 
\item $\fbmsgO_n$ has genus zero and exactly $n$ pairwise isometric boundary components. 
\item The area of $\fbmsgO_n$ is strictly between $\pi$ and $2\pi$. 
\item The sequence $\{\fbmsgO_n\}_{n\geq2}$ converges in the sense of varifolds to the flat, equatorial disc with multiplicity two as $n\to\infty$.  
\end{itemize}
\end{theorem}

\begin{remark}
Conjecturally the full symmetry group of $\fbmsgO_n$ is the prismatic group $\pri_n$ of order $4n$ (cf. \cite[§\,2]{CSWnonuniqueness}). 
We prefer to work with the subgroup $\dih_n$ because $\pri_n$ is not orientation-preserving being generated by the reflection across the plane $\{y=0\}$ in addition to the two rotations $\rotation_{x\text{-axis}}^{\pi}$ and $\rotation_{z\text{-axis}}^{2\pi/n}$. 
We also conjecture that $\fbmsgO_n$ coincides for all $n$ with the surface $\Gamma_n^{\mathrm{FPZ}}$ described in \cite[Conjecture 7.8]{CSWnonuniqueness} which, in turn, coincides with the genus zero free boundary minimal surfaces constructed by Folha--Pacard--Zolotareva \cite{FolhaPacardZolotareva2017} for all sufficiently large $n\in\N$.
\end{remark}

\begin{remark}
For $n=3$, Theorem \ref{thm:fbmsg0} states the existence of a \emph{free boundary minimal trinoid} as visualized in Figure \ref{fig:trinoid}, left image. 
For $n=2$, Theorem \ref{thm:fbmsg0} provides a variational construction of an embedded, $\dih_2$-equivariant free boundary minimal annulus in $\B^3$. 
We are not aware of any result stating that such an object must be isometric to the critical catenoid even though the uniqueness of the latter has been conjectured and proved under various other symmetry assumptions \cite{McGrath2018,KusnerMcGrath2020}.  
\end{remark}

\begin{figure}%
\centering   
\begin{unitball}
\FBMS{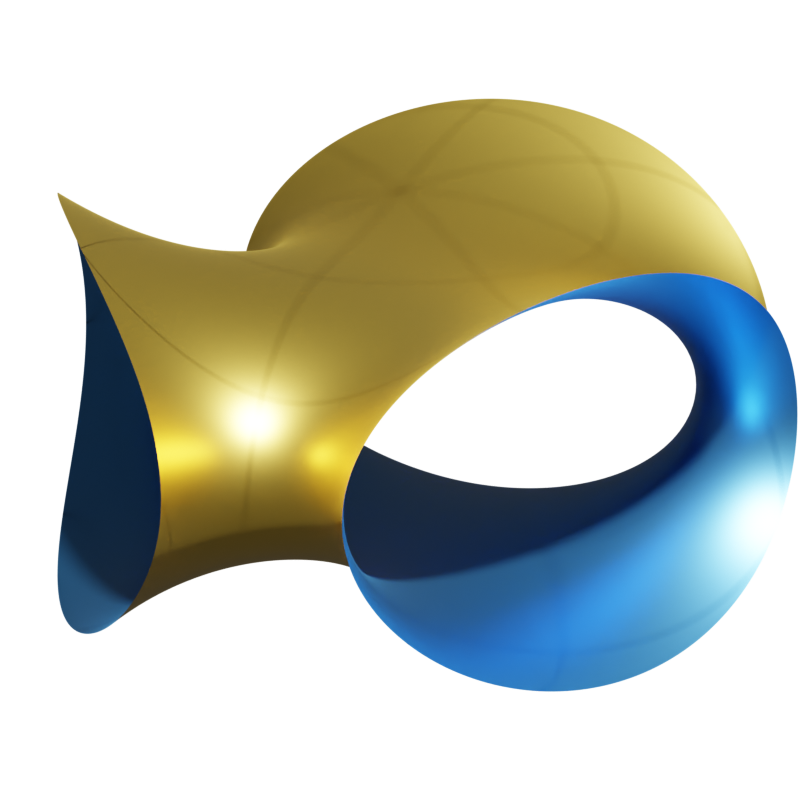}
\end{unitball}
\hfill
\begin{unitball}
\FBMS{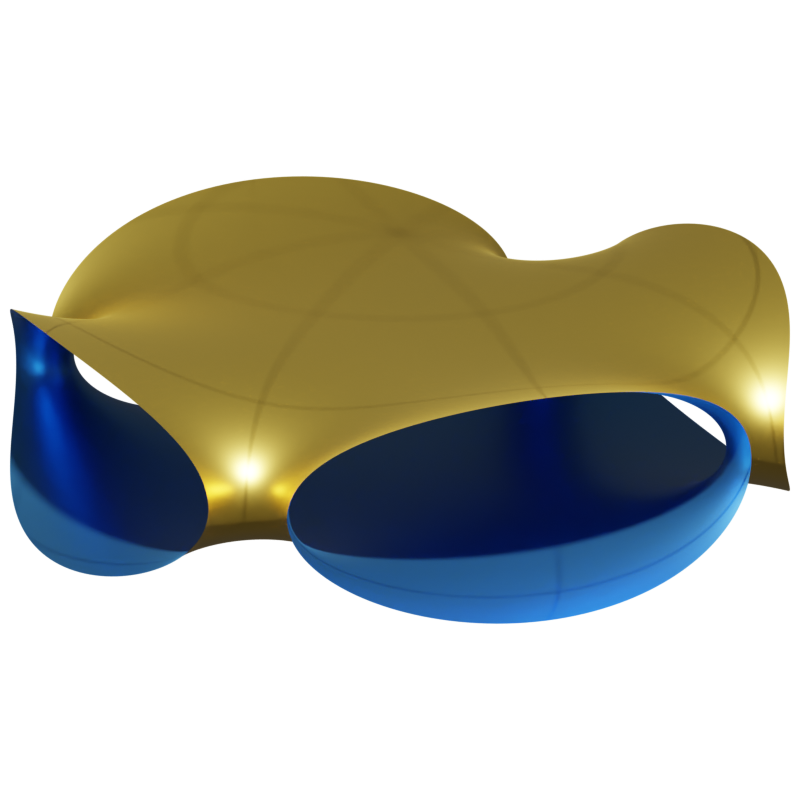}
\end{unitball}
\quad~
\caption{Simulations of the free boundary minimal surfaces $\fbmsgO_3$ and $\fbmsgO_4$ in $\B^3$.}%
\label{fig:trinoid}%
\end{figure}

The variational proof of Theorem \ref{thm:fbmsg0} involves the following steps. 
\begin{enumerate}[label={(\Roman*)}]
\item\label{thm:fbmsg0-Step1} Design a \emph{sweepout} of $\B^3$ consisting of $\dih_n$-equivariant surfaces with genus zero, $n$ boundary components and area strictly less than $2\pi$ and define its $\dih_n$-equivariant saturation.  
\item\label{thm:fbmsg0-Step2} Verify the \emph{width} estimate and apply equivariant min-max theory to extract a min-max sequence converging in the sense of varifolds to a free boundary minimal surface $\Gamma$. 
\item\label{thm:fbmsg0-Step3} Control the \emph{topology}: Show that $\Gamma$ has genus zero and exactly $n$ boundary components.
\item\label{thm:fbmsg0-Step4} Determine the \emph{asymptotic behaviour} of $\{\fbmsgO_n\}_{n\geq2}$ for $n\to\infty$ in the sense of varifolds.  
\end{enumerate}

In \cite[§\,5]{Ketover2016FBMS}, Ketover provides the details for step~\ref{thm:fbmsg0-Step1}. 
For the sake of completeness, we repeat the construction of the sweepout using similar notation as in \cite[§\,2]{CarlottoFranzSchulz2022}. 
The min-max theory used in step~\ref{thm:fbmsg0-Step2} has also been developed in \cite{Ketover2016FBMS}.
Regarding step \ref{thm:fbmsg0-Step3}, we recall that varifold convergence is too weak to preserve the topology in general, and in this specific case, boundary components could be lost or gained in the limit.   
The idea presented in \cite{Ketover2016FBMS} is to classify all possible topological changes along the min-max sequence in question and to conclude that they would always result in two equatorial discs contradicting a strict upper area bound.  
We provide an alternative, more general approach based on a combination of Theorems~\ref{thm:main} and~\ref{thm:mainorientable} and the following structural lemma about arbitrary (not necessarily minimal) equivariant surfaces of genus zero. 

\begin{lemma}\label{lem:structural}
Given $3\leq n\in\N$, let $M\subset\R^3$ be any convex, bounded domain with piecewise smooth boundary such that the cyclic group $\Z_{n}$ acts on $M$ by isometries.  
Let $\Sigma\subset M$ be any compact, connected, $\Z_{n}$-equivariant surface of genus zero which is properly embedded and has $b\in\{2,3,\ldots,n\}$ boundary components. 
Then $b\in\{2,n\}$. 
Moreover, if $\Sigma$ intersects the singular locus of the $\Z_n$-action, then $b=n$ and the group $\Z_n$ acts simply transitively on the collection of boundary components. 
\end{lemma}

\begin{proof}
We claim that $\Sigma$ can intersect the axis of rotation at most twice. 
As described e.\,g. in the proof of \cite[Corollary D.2]{CSWnonuniqueness}, every boundary component of $\Sigma$ can be closed up by gluing in a topological disc while preserving embeddedness and $\Z_n$-equivariance. 
The resulting surface $\tilde{\Sigma}$ is closed and embedded, thus two-sided, allowing a global unit normal vector field $\nu$ which inherits the $\Z_n$-equivariance. 
Let $\xi_0\subset\R^3$ denote the singular locus of the $\Z_n$-action, i.\,e. the axis of rotation. 
If $\tilde{\Sigma}$ intersects $\xi_0$ in some point $p_0$, then $\nu(p_0)\in\xi_0$ since $p_0$ is fixed by the group action. 
Therefore, any intersection of $\tilde{\Sigma}$ with the axis $\xi_0$ must be orthogonal and the number of such intersections is well-defined and finite. 
Moreover, the number of such intersections is even, i.\,e.~equal to $2j$ for some integer $j\geq0$, because $\tilde{\Sigma}$ is compact without boundary.   
The quotient $\tilde\Sigma/\Z_{n}$ has an orbifold structure (see \cite[Prop.~13.2.1]{Thurston2022})
and \cite[Prop.~13.3.1]{Thurston2022} implies that the underlying space is an orientable topological surface $\tilde\Sigma'$ without boundary. 
Thus, $\tilde\Sigma'$ has Euler characteristic $\chi(\tilde\Sigma')=2-2g'$ for some nonnegative integer $g'$. 
A variant of the Riemann--Hurwitz formula (see e.\,g.~\cite[§\,IV.3]{Freitag2011}) implies
\begin{align*} 
2=\chi(\tilde\Sigma)&=n\chi(\tilde\Sigma')-2j(n-1)
=2n-2ng'-2j(n-1)
\end{align*}
or equivalently $0=n-ng'-j(n-1)-1$ where all variables are nonnegative integers. 
Consequently, $g'=0$ and $j=1$ which means that $\tilde\Sigma$ intersects $\xi_0$ exactly twice and $\Sigma$ intersects $\xi_0$ at most twice. 

Let $i\in\{0,1,2\}$ be the number of intersections of the original surface $\Sigma$ with the axis of rotation. 
The quotient $\Sigma'=\Sigma/\Z_{n}$ is again an orientable topological surface with Euler characteristic $\chi(\Sigma')=2-b'$ for some nonnegative integer $b'$.
The Riemann--Hurwitz formula now implies
\begin{align*} 
2-b=\chi(\Sigma)&=n\chi(\Sigma')-i(n-1)
=2n-b'n-in+i
\end{align*}
or equivalently, $b=(b'-(2-i))n+(2-i)$. 
Since $n\geq3$ and $(2-i)\in\{0,1,2\}$ as shown above, we may complete the proof by distinguishing the following cases. 
\begin{itemize}  
\item $b'<(2-i)\Rightarrow b<0$, which is absurd. 
\item $b'>(3-i)\Rightarrow b>n$, which contradicts our assumption $b\in\{2,3,\ldots,n\}$. 
\item $b'=(2-i)\Rightarrow b=(2-i)$. 
In this case the assumption $b\geq 2$ yields $i=0$ and $b=2$. 
\item $b'=(3-i)\Rightarrow b=n+(2-i)$. 
Then the assumption $b\leq n$ yields $i=2$ and $b=n$.  
\end{itemize}
In the last case, we also have $b'=1$ which implies that the quotient $\Sigma/\Z_n$ has exactly one boundary component and its orbit is $\partial\Sigma$. 
Consequently, $\Z_n$ acts simply transitively on the connected components of $\partial\Sigma$.  
\end{proof}

\begin{proof}[Proof of Theorem \ref{thm:fbmsg0}]
\ref{thm:fbmsg0-Step1} \emph{Sweepout construction.} 
For every $k\in\{1,\ldots,n\}$ let $B_\varepsilon(p_k)\subset\R^3$ denote the ball of radius $\varepsilon>0$ around the equatorial point 
\begin{align}\label{eqn:pk}
p_k=\biggl(\cos\Bigl(\frac{2\pi k}{n}\Bigr),\;\sin\Bigl(\frac{2\pi k}{n}\Bigr),\;0\biggr)
\end{align}
and consider the subsets 
\begin{align*}
D_{\varepsilon}&\vcentcolon=\B^3\cap\{z=0\}\setminus\bigcup^{n}_{k=1}B_\varepsilon(p_k), & 
D_{t,\varepsilon}&\vcentcolon=\bigl(\sqrt{1-t^2}D_{\varepsilon}\bigr)+(0,0,t). 
\end{align*}
Given any $t\in\interval{-1,1}$ and any $0<\varepsilon<\sin(\pi/n)$ the set $D_{t,\varepsilon}$ is a topological disc inside $\B^3$ 
and we may $\dih_n$-equivariantly connect the two sets $D_{\pm t,\varepsilon}$ by means of $n$ ribbons. 
To be precise, with $0<\varepsilon_0\leq t_0\ll1$ to be chosen, we let $\varepsilon\colon\Interval{t_0,1}\to\intervaL{0,\varepsilon_0}$ be a continuous function of $t$ such that $\varepsilon(t)\to0$ as $t\to1$ and define  
\begin{align}\label{eqn:20230615}
\Omega_{t}&\vcentcolon=\bigcup_{\tau\in[-t,t]}D_{\tau ,\varepsilon(t)}, & 
\Sigma_{t}&\vcentcolon=\overline{\partial\Omega_{t}\setminus\partial\B^3}
\end{align}
for all $t\in\Interval{t_0,1}$, where $\partial\Omega_{t}$ refers to the topological boundary of the set $\Omega_{t}\subset\R^3$. 
Naturally, we define $\Sigma_1$ to be the union of great circles on $\partial\B^3$ connecting the equatorial points $p_1,\ldots,p_n$ defined in \eqref{eqn:pk} with the north and the south pole  
(see Figure \ref{fig:sweepout}, images 1--2).
The area of each ribbon connecting the two subsets $\Sigma_t\cap\{\abs{z}=t\}$ is bounded from above by $2\pi\varepsilon_0t$ (cf. \cite[Lemma~2.4]{CarlottoFranzSchulz2022}). 
Hence, for all $t\in\Interval{t_0,1}$ the area of $\Sigma_t$ satisfies 
\(
\hsd^2(\Sigma_t)\leq2\pi(1-t^2+n\varepsilon_0t).
\) 
If we choose $\varepsilon_0=t_0/(2n)$ then 
\begin{align}\label{eqn:upper-area-bound}
\hsd^2(\Sigma_t)\leq (2-t_0^2)\pi <2\pi
\end{align}
for all $t\in\Interval{t_0,1}$. 
Setting $p_0=p_n$, let 
\[
\Sigma_0\vcentcolon=\bigcup_{k=1}^{n}\bigl\{r(p_{k-1}+p_{k})\st r>0\bigr\}\cap\B^3.
\] 
As one decreases $t$ further from $t_0$ to $0$, the idea is to deform $\Sigma_t$ continuously into $\Sigma_0$ without violating the strict upper area bound \eqref{eqn:upper-area-bound}. 
This can be achieved by appealing to the catenoid estimate, \cite[Proposition~2.1 and Theorem~2.4]{KetoverMarquesNeves2020},
which in our specific case can be carried out explicitly as follows. 
Given $0<r<\sin(\pi/n)$ and $0<h<e^{-8n}$ we consider the surfaces 
\begin{align*}
C_s^{r,h}\vcentcolon=\biggl\{(x,y,z)\in\R^3\st\sqrt{x^2+y^2}=\frac{r\cosh(sz)}{\cosh(sh)},~\abs{z}\leq h\biggr\} 	
\end{align*}
as in \cite[(7)]{CarlottoFranzSchulz2022}
which interpolate between the cylinder $C_0^{r,h}$ with radius $r$ and height $2h$ and the union $C_\infty^{r,h}$ of two parallel discs with radius $r$ connected by a vertical line segment of length $2h$. 
For each $k\in\{1,\ldots,n\}$ and each $0\leq s\leq\infty$ let $C_{s,k}^{r,h}\vcentcolon=C_s^{r,h}+p_k$ be the horizontally translated surface centered at the equatorial point $p_k$. 
By \cite[Lemma A.1]{CarlottoFranzSchulz2022} the slice of maximal area in the family $\{C_{s,k}^{r,h}\}_{s>0}$ of surfaces is an unstable catenoid. 
In particular, this implies that, if $r$, $h$ and $r/h$ are sufficiently small, then 
(cf.~\cite[(10)]{CarlottoFranzSchulz2022} and \cite[Proposition~2.1]{KetoverMarquesNeves2020}) 
\begin{align*}
\sup_{s\geq0}
\hsd^2\bigl(C_{s,k}^{r,h}\cap\B^3\bigr)
&\leq \hsd^2\bigl(C_{\infty,k}^{r,h}\cap\B^3\bigr)+\frac{4\pi h^2}{(-\log h)}.
\end{align*}

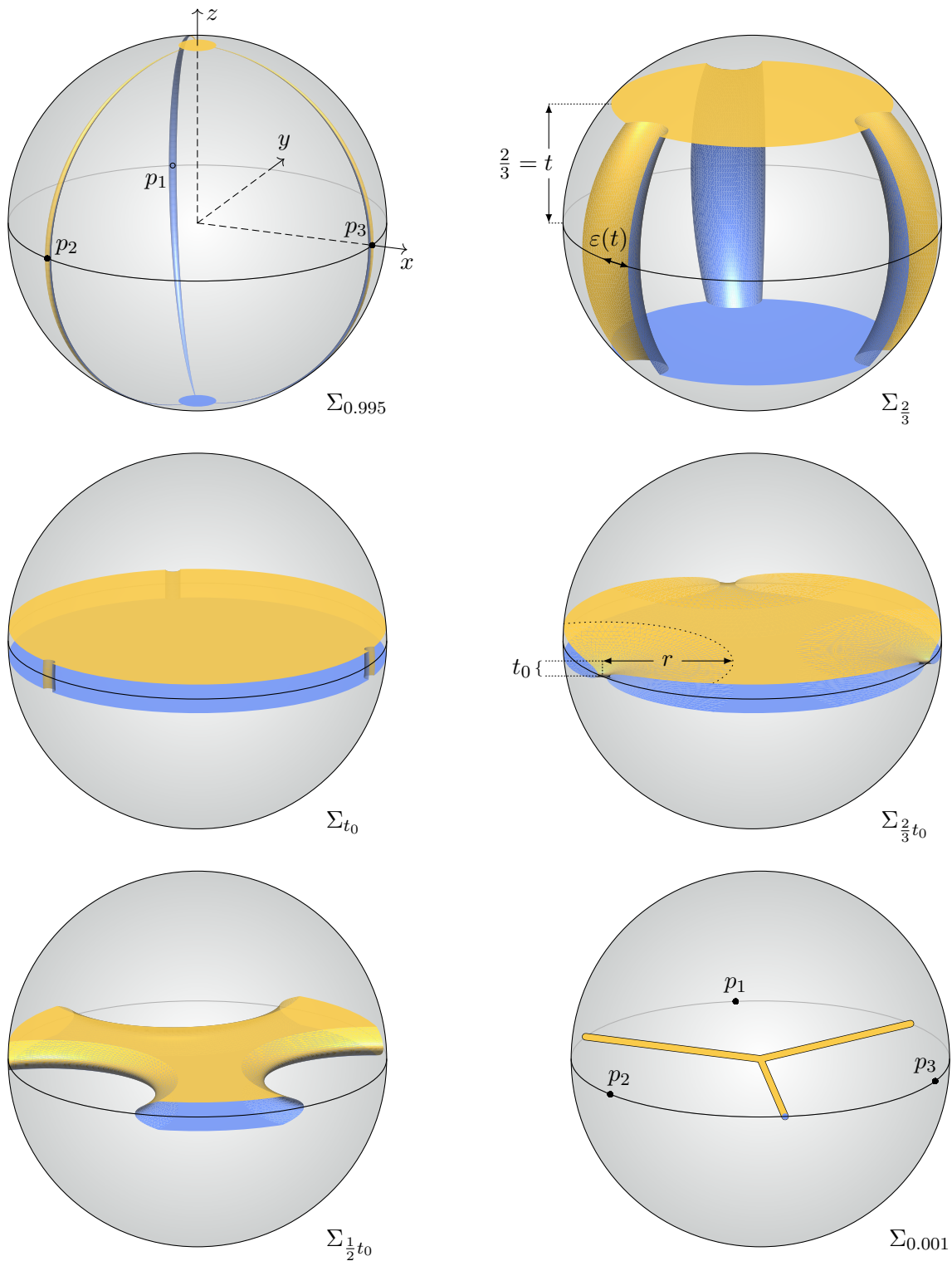
\begin{figure}%
\providecommand{\backsphere}{\path[tdplot_screen_coords,inner color=white,outer color=cyan!5!black!20](0,0)circle(1);\tdplotdrawarc[black!30,thin]{(0,0,0)}{1}{0}{360}{}{}}
\providecommand{\frontsphere}{\tdplotdrawarc[black,thin]{(0,0,0)}{1}{\phiO-180}{\phiO}{}{}\draw[tdplot_screen_coords,black,semithick] (0,0) circle (1);}
\pgfmathsetmacro{\thetaO}{72}
\pgfmathsetmacro{\phiO}{90-3*45/2} 
\tdplotsetmaincoords{\thetaO}{\phiO}
\begin{tikzpicture}[tdplot_main_coords,line cap=round,line join=round,scale=\unitscale,baseline={(0,0,0)}]
\backsphere
\draw[densely dashed]
(0,0,0)--(0.995,0,0)(1,0,0)coordinate(x)
(0,0,0)--(0,1,0)coordinate(y)
(0,0,0)--(0,0,0.9950)coordinate(z)
;
\FBMS{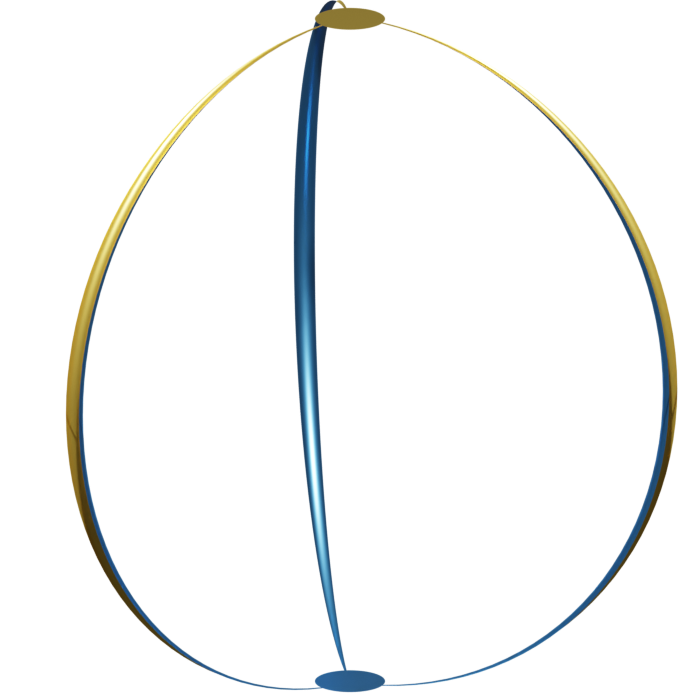}
\draw[->](x)--(1.2,0,0)node[below]{$x$};
\draw[->](y)--(0,1.2,0)node[above]{$y$};
\draw[->](z)--(0,0,1.2)node[below right,inner sep=0]{~$z$};
\draw (120:1)coordinate[draw,circle,inner sep=0.8pt]node[anchor=45]{$p_1$};
\draw plot[bullet](240:1)node[anchor=-150]{$p_2$};
\draw plot[bullet](360:1)node[anchor=-45]{$p_3$};
\draw[tdplot_screen_coords](FBMS.south east)+(-2ex,.6ex)
node[below right,inner sep=0]{$\mathrlap{\Sigma_{0.995}}$};
\frontsphere
\end{tikzpicture}
\hfill
\begin{tikzpicture}[tdplot_main_coords,line cap=round,line join=round,scale=\unitscale,baseline={(0,0,0)}]
\backsphere
\FBMS{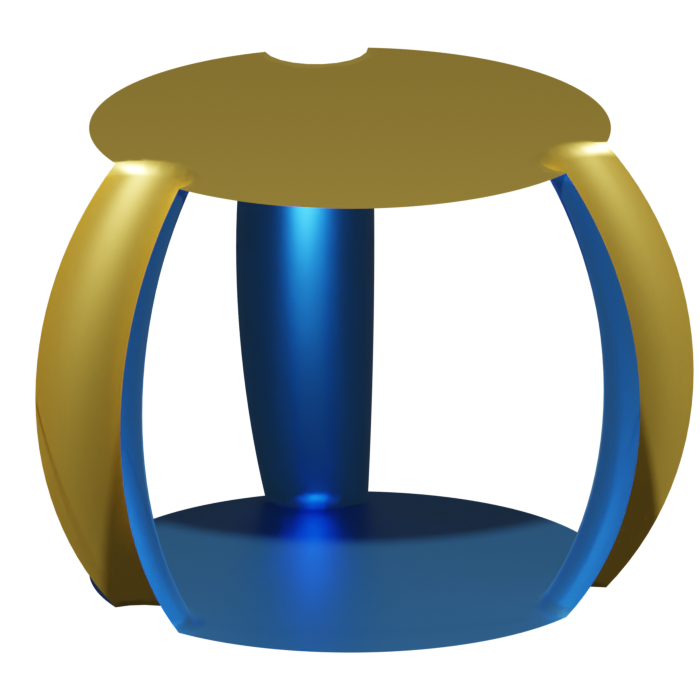}
\draw({sqrt(5/9)*cos(\phiO-180)},{sqrt(5/9)*sin(\phiO-180)},2/3)coordinate(z);
\draw[tdplot_screen_coords,densely dotted](z)++(-0.33,0)coordinate(top)--(z);
\path[tdplot_screen_coords](top)--++(0,{-2*sin(\thetaO)/3})
node[midway,inner sep=0](t){$\mathllap{\frac{2}{3}={}}t$}
coordinate(bot);
\draw[-latex](t)--(top);
\draw[-latex](t)--(bot);
\draw[tdplot_screen_coords,densely dotted](bot)--(-1,0);
\draw[latex-latex](240:1)--(240+0.2*180/3:1)node[near start,above]{$\varepsilon(t)$};
\draw[tdplot_screen_coords](FBMS.south east)+(-2ex,.6ex)
node[below right,inner sep=0]{$\mathrlap{\Sigma_{\frac{2}{3}}}$};
\frontsphere
\end{tikzpicture}
\bigskip 

\begin{tikzpicture}[tdplot_main_coords,line cap=round,line join=round,scale=\unitscale,baseline={(0,0,0)}]
\backsphere
\FBMS{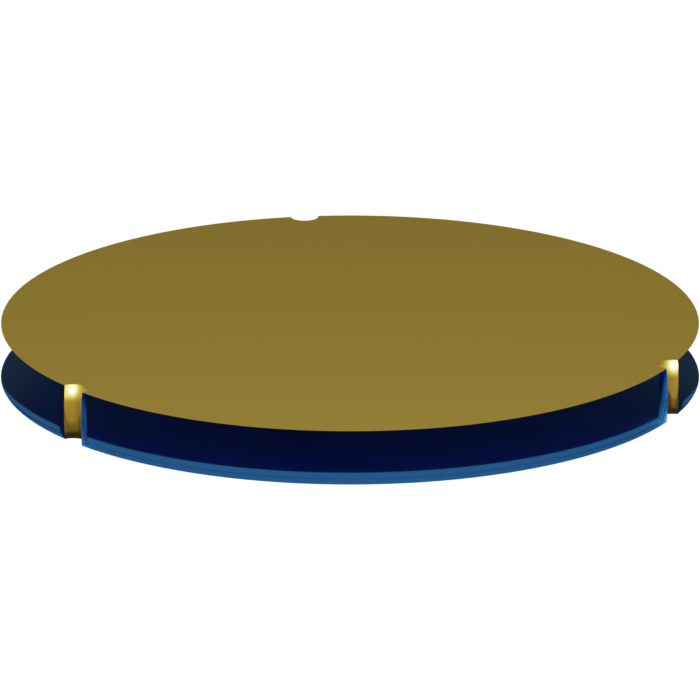}
\draw[tdplot_screen_coords](FBMS.south east)+(-2ex,.6ex)
node[below right,inner sep=0]{$\mathrlap{\Sigma_{t_0}}$};
\frontsphere
\end{tikzpicture}
\hfill
\begin{tikzpicture}[tdplot_main_coords,line cap=round,line join=round,scale=\unitscale,baseline={(0,0,0)}]
\backsphere
\FBMS{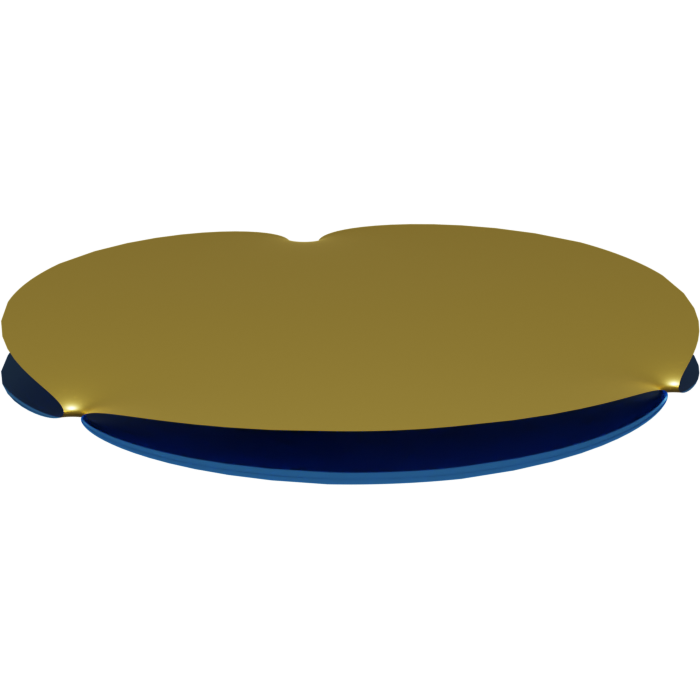}
\begin{scope}
\clip(0,0,0.08)circle(1);
\draw[dotted](-0.8437,-1.4613,0.0800)arc(240:240+360:0.6896)coordinate[pos=0.4](r);
\end{scope}
\path[tdplot_screen_coords](r)--++(-0.6896,0)node[midway](m){$r$}coordinate(l);
\draw[tdplot_screen_coords,-latex](m)--(r);
\draw[tdplot_screen_coords,-latex](m)--(l);
\draw[tdplot_screen_coords,densely dotted](l)--++(0,-0.08)coordinate(u);
\draw[tdplot_screen_coords,densely dotted]
(l)--++(-0.3,0)coordinate(top)
(u)--++(-0.3,0)coordinate(bot);
\path[tdplot_screen_coords](top)--(bot)node[left=2.5pt,midway]{$t_0$}
node[left=1pt,midway,inner sep=0,scale=0.9]{$\scriptstyle\{$};
;
\draw[tdplot_screen_coords](FBMS.south east)+(-2ex,.6ex)
node[below right,inner sep=0]{$\mathrlap{\Sigma_{\frac{2}{3}t_0}}$};
\frontsphere
\end{tikzpicture}
\bigskip 

\begin{tikzpicture}[tdplot_main_coords,line cap=round,line join=round,scale=\unitscale,baseline={(0,0,0)}]
\backsphere
\FBMS{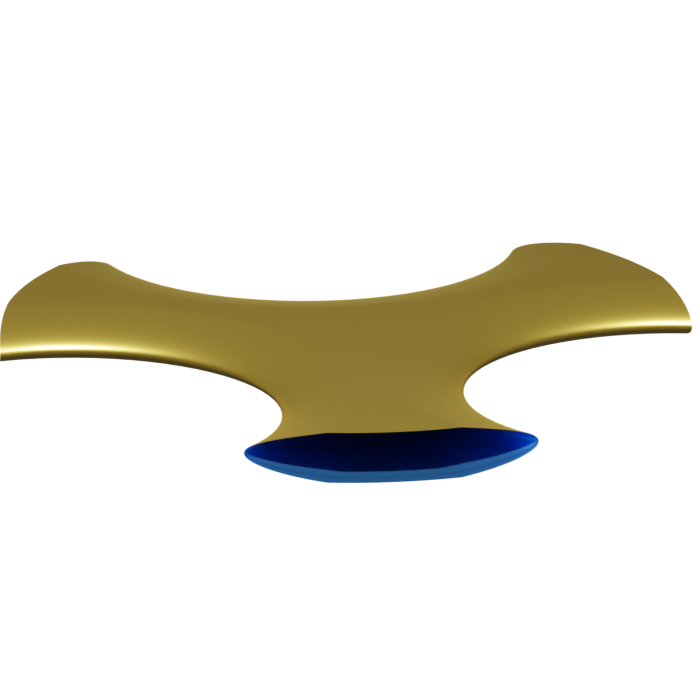}
\draw[tdplot_screen_coords](FBMS.south east)+(-2ex,.6ex)
node[below right,inner sep=0]{$\mathrlap{\Sigma_{\frac{1}{2}t_0}}$};
\frontsphere
\end{tikzpicture}
\hfill
\begin{tikzpicture}[tdplot_main_coords,line cap=round,line join=round,scale=\unitscale,baseline={(0,0,0)}]
\backsphere
\FBMS{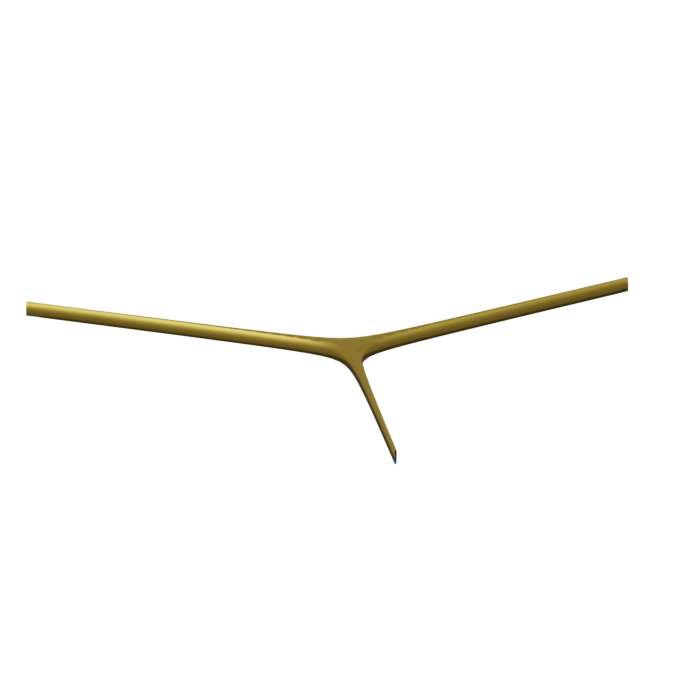}
\draw plot[bullet](120:1)node[above]{$p_1$};
\draw plot[bullet](240:1)node[anchor=-120]{$p_2$};
\draw plot[bullet](360:1)node[anchor=-60]{$p_3$};
\draw[tdplot_screen_coords](FBMS.south east)+(-2ex,.6ex)
node[below right,inner sep=0]{$\mathrlap{\Sigma_{0.001}}$};
\frontsphere
\end{tikzpicture}
\caption{Successive stages of the sweepout in the case $n=3$.}%
\label{fig:sweepout}%
\end{figure}%

We are now free to choose $t_0=h<e^{-8n}$ and $s=s_0$ such that $r/\cosh(s_0t_0)=\varepsilon(t_0)$. 
For each $k\in\{1,\ldots,n\}$ we deform $\Sigma_{t_0}\cap\{(x,y,z)\in\B^3\st \dist((x,y,0),p_k)<r\}$
into a copy of $C_{s_0,k}^{r,t_0}\cap\B^3$. 
In fact, the two surfaces in question are arbitrarily close provided that both $\varepsilon(t_0)\in\interval{0,\varepsilon_0}$ (and thus $1/s_0>0$) are sufficiently small, such that we can continuously deform one into the other without significantly increasing the area. 
This deformation (performed simultaneously and equivariantly for all $k\in\{1,\ldots,n\}$) defines $\Sigma_t$ for $t\in\Interval{2t_0/3,t_0}$ (see Figure~\ref{fig:sweepout}, images 3--4).
As one decreases $t$ further form $2t_0/3$ to $t_0/3$, we decrease $s$ from $s_0$ to $0$ and define $\Sigma_t$ accordingly such that (cf. \cite[(13)]{CarlottoFranzSchulz2022})
\begin{align*} 
\hsd^2(\Sigma_t)&\leq\hsd^2(\Sigma_{t_0})+\frac{4\pi t_0^2\,n}{(-\log t_0)}
\leq\Bigl(2- \frac{t_0^2}{2}\Bigr)\pi 
\end{align*}
for all $t\in\Interval{t_0/3,2t_0/3}$. 
While $t$ decreases further from $t_0/3$ to $0$ it is now straightforward to deform $\Sigma_t$ into $\Sigma_0$ with area decreasing monotonically to zero (see Figure~\ref{fig:sweepout}, images 5--6). 
To conclude the construction of the sweepout $\{\Sigma_t\}_{t\in[0,1]}$ we regularize the slices equivariantly without violating the strict $2\pi$ upper area bound such that $\Sigma_t$ is a smooth surface for $0<t<1$.

\ref{thm:fbmsg0-Step2} \emph{Width estimate.}
Given the sweepout $\{\Sigma_t\}_{t\in[0,1]}$ constructed in \ref{thm:fbmsg0-Step1}, we 
consider the width $W_\Pi$ of the $\dih_n$-saturation $\Pi$, as introduced in Definition~\ref{def:SaturationWidth}, and claim that 
\begin{align}\label{eqn:width-estimate}
\pi<W_\Pi<2\pi.
\end{align}
The upper bound in \eqref{eqn:width-estimate} follows directly from the fact that $\hsd^2(\Sigma_t)<2\pi$ for all $t\in[0,1]$. 
For the lower bound we define a family $\{F_t^{\Sigma}\}_{t\in[0,1]}$ of $\dih_n$-equivariant subsets of $\B^3$ with finite perimeter such that $\Sigma_t$ is the relative boundary of $F_t^{\Sigma}$ in $\B^3$ and such that $F_t^{\Sigma}$ does not contain any of the points $p_1,\ldots,p_n$ defined in \eqref{eqn:pk} for all $t\in\interval{0,1}$. 
(In fact, $F_t^{\Sigma}$ roughly corresponds to the set $\Omega_t$ defined explicitly in \eqref{eqn:20230615}, at least for all $t\in\Interval{t_0,1}$.)   
We recall that given any $\{\Lambda_t\}_{t\in[0,1]}\in\Pi$ there exists a smooth map $\Phi\colon[0,1]\times\B^3\to\B^3$, where $\Phi(t,\cdot)$ is a diffeomorphism which commutes with the $\dih_n$-action for all $t\in[0,1]$ and coincides with the identity for all $t\in\{0,1\}$, such that $\Lambda_t=\Phi(t,\Sigma_t)$ for all $t\in[0,1]$. 
Then, $F_t\vcentcolon=\Phi(t,F_t^{\Sigma})$ is $\dih_n$-equivariant such that $\hsd^3(F_0)=0$ and $\hsd^3(F_1)=\hsd^3(\B^3)$, and such that for every $t\in\interval{0,1}$
\begin{itemize}
\item $\Lambda_t\setminus\partial\B^3=\partial F_t\setminus\partial\B^3$, i.\,e.~$\Lambda_t$ is the relative boundary of $F_t$ in $\B^3$;
\item $(s\to t)\Rightarrow \hsd^3(F_s\bigtriangleup F_{t})\to0$, where we recall that $F_s\bigtriangleup F_{t}\vcentcolon=(F_{s}\setminus F_{t})\cup(F_{t}\setminus F_{s})$;
\item $\{p_1,\ldots,p_n\}\subset\B^3\setminus F_{t}$. 
\end{itemize} 
In particular, the function $[0,1]\ni t\mapsto \hsd^3(F_t)$ is continuous and there exists $t_*\in\interval{0,1}$ such that $\hsd^3(F_{t_*})=\frac{1}{2}\hsd^3(\B^3)$. 
By the isoperimetric inequality (cf. \cite[Satz~1]{BokowskiSperner1979}, \cite[Theorem~5]{Ros2005}), we obtain $\hsd^2(\Lambda_{t_*})\geq\pi$ and hence $W_\Pi\geq\pi$. 

To prove that the lower bound on the width is strict, we recall the stability of the isoperimetric inequality 
(e.\,g. from \cite[Lemma~3.6]{CarlottoFranzSchulz2022}), 
which implies that there exists $\delta_0>0$ such that if $\hsd^2(\Lambda_{t_*})<\pi+\delta_0$, 
then there exists a half-ball $\Omega=\{p\in\B^3\st p\cdot v\geq0\}$ given by some $v\in\Sp^2$ satisfying $\phi(v)=\pm v$ for all $\phi\in\dih_n$, such that $\hsd^3(F_{t_*}\bigtriangleup\Omega)\leq\pi/6$.  
If $n\geq3$ then necessarily $v=\pm(0,0,1)$. 
In the case $n=2$ we could additionally have $v=\pm(1,0,0)$ or $v=\pm(0,1,0)$. 
In any case, we can find some $\varphi\in\dih_n$ such that $\varphi(v)=-v$ or equivalently, $\varphi(\Omega)=\B^3\setminus\Omega$. 
Recalling that the set $\{p_1,\ldots,p_n\}\subset \B^3\setminus F_{t_*}$ is fixed under the action of $\dih_n$, we necessarily have $\varphi(F_{t_*})=F_{t_*}$. 
Since $\varphi\colon\B^3\to\B^3$ is injective,  
\begin{align*}
\varphi(F_{t_*}\bigtriangleup\Omega)
&=\bigl(\varphi(F_{t_*})\setminus\varphi(\Omega)\bigr)\cup\bigl(\varphi(\Omega)\setminus\varphi(F_{t_*})\bigr)  
\\&
=(F_{t_*}\setminus\Omega^{\complement})\cup(\Omega^{\complement}\setminus F_{t_*}) 
\\
&=(F_{t_*}\cap\Omega)\cup(\Omega^{\complement}\cap F_{t_*}^{\complement})
\\&
=(F_{t_*}\cup\Omega^{\complement})\cap(\Omega\cup F_{t_*}^{\complement})
\\
&=(\Omega\setminus F_{t_*})^{\complement}\cap(F_{t_*}\setminus\Omega)^{\complement}
=(F_{t_*}\bigtriangleup\Omega)^{\complement},
\end{align*}
where $A^{\complement}\vcentcolon=\B^3\setminus A$ for any $A\subset\B^3$. 
Since $\varphi\colon\B^3\to\B^3$ is also an isometry, we conclude  
$\hsd^3\bigl(F_{t_*}\bigtriangleup\Omega\bigr)=\hsd^3\bigl((F_{t_*}\bigtriangleup\Omega)^{\complement}\bigr)$, 
which leads to the contradiction 
\begin{align*}
\tfrac{4}{3}\pi=\hsd^3\bigl(\B^3\bigr)
=\hsd^3\bigl(F_{t_*}\bigtriangleup\Omega\bigr)
+\hsd^3\bigl((F_{t_*}\bigtriangleup\Omega)^{\complement}\bigr)
=2\hsd^3\bigl(F_{t_*}\bigtriangleup\Omega\bigr)\leq\tfrac{1}{3}\pi.
\end{align*}
Consequently, we must have $\hsd^2(\Lambda_{t_*})\geq\pi+\delta_0$ and since $\{\Lambda_t\}_{t\in[0,1]}\in\Pi$ is arbitrary, the claim $W_\Pi\geq\pi+\delta_0$ follows. 

\begin{remark}
The stability of the isoperimetric inequality has also been used in \cite[Proposition~3.7]{CarlottoFranzSchulz2022} and in \cite[Proposition~4.2]{Ketover2016FBMS} to prove a strict width estimate. 
However, in both cases \emph{every} slice $\Lambda_t$ of the sweepouts in question divides $\B^3$ into two sets $F_t$ and $F_t^{\complement}$ of equal volume and -- unlike in our case -- there exists $\varphi\in\dih_n$ such that $\varphi(F_t)=F_t^{\complement}$. 
Therefore, the details in \ref{thm:fbmsg0-Step2} differ from those in the aforementioned references. 
\end{remark}

By Theorem~\ref{thm:PreviousResults}, the lower bound $W_\Pi>0=\max\{\hsd^2(\Sigma_0),\hsd^2(\Sigma_1)\}$ suffices to extract a min-max sequence $\{\Sigma^j\}_{j\in\N}$ consisting of $\dih_n$-equivariant surfaces converging in the sense of varifolds to $m\fbmsgO_n$, where $\fbmsgO_n$ is a compact, connected, embedded, $\dih_n$-equivariant free boundary minimal surface in $\B^3$ and where the multiplicity $m$ is a positive integer. 
(We indeed obtain just one connected component $\fbmsgO_n$ by \cite[Lemma 2.4]{FraserLi2014}, as $\B^3$ is a simply connected with nonnegative Ricci curvature and strictly convex boundary.)
Moreover, $m\hsd^2(\fbmsgO_n)=W_\Pi$.

\ref{thm:fbmsg0-Step3} \emph{Topological control.}
In step \ref{thm:fbmsg0-Step2} we proved the strict inequalities $\pi<m\hsd^2(\fbmsgO_n)<2\pi$. 
Since any free boundary minimal surface in $\B^3$ has at least area $\pi$ by \cite[Theorem~5.4]{FraserSchoen2011}, we directly obtain $m=1$. 
Moreover, the strict lower bound $\pi<\hsd^2(\fbmsgO_n)$ implies that $\fbmsgO_n$ is not isometric to the flat, equatorial disc.
By \cite{Nitsche1985}, the equatorial disc is the only free boundary minimal disc in $\B^3$ up to ambient isometries; hence $\fbmsgO_n$ is not a topological disc. 
Being properly embedded in $\B^3$, the surface $\fbmsgO_n$ is orientable and we have $\genus(\fbmsgO_n)=\gsum(\fbmsgO_n)\leq 0$ by Theorem~\ref{thm:main}, recalling that every surface in our sweepout has genus zero.  
Thus, $\fbmsgO_n$ has genus zero and since it is not a topological disc, $\fbmsgO_n$ must have at least two boundary components.  
By Theorem \ref{thm:mainorientable} we have 
\(
\bsum(\fbmsgO_n)\leq n-1		
\)
where we used that $\gsum(\fbmsgO_n)=0=\gsum(\Sigma^j)$ and $\bsum(\Sigma^j)=\beta_0(\partial\Sigma^j)-1=n-1$ for all $j\in\N$. 
Consequently, $\fbmsgO_n$ (being connected) has at most $n$ boundary components. 
In the case $n=2$ the claim follows. 
Therefore, let us assume $n\geq3$ for the rest of the proof.  
Lemma~\ref{lem:structural} then implies that the number of boundary components for $\fbmsgO_n$ is either $2$ or $n$.

\emph{Ruling out the annulus.}
Towards a contradiction, suppose that $\fbmsgO_n$ has only two boundary components. 
In this case, Lemma \ref{lem:structural} implies that $\fbmsgO_n$ is disjoint from the singular locus $\xi_0\vcentcolon=\B^3\cap\{x=y=0\}$ of the $\Z_n$-action.  
Let $\varepsilon>0$ such that $U_{3\varepsilon}\fbmsgO_n$ is still disjoint from $\xi_0$. 
For each $j\geq J_\varepsilon$, let $\tilde\Sigma^j\subset U_{\varepsilon}\fbmsgO_n$ be the surface obtained from $\Sigma^j$ through surgery as constructed in Theorem \ref{thm:TopologyOfAlmMinSeq}. 
Note that $\{\tilde\Sigma^j\}_{j\geq J_\varepsilon}$ inherits the property of being $\dih_n$-almost minimizing in sufficiently small annuli from $\{\Sigma^j\}_{j\in\N}$, as shown in the beginning of the proof of Theorem~\ref{thm:TopologyOfAlmMinSeq}.

Let $M$ be the smooth ambient manifold which is obtained by regularizing the quotient $\B^3/\Z_n$ in an $\varepsilon$-neighbourhood of the singular locus $\xi_0$. 
The regularization at the poles can be done such that $\partial M$ is strictly mean convex.  
Since the surfaces $\fbmsgO_n$ and $\tilde{\Sigma}^j$ are $\Z_n$-equivariant and disjoint from $U_\varepsilon\xi_0$, 
their quotients $\hat{\fbmsgO}_n=\fbmsgO_n/\Z_n$ and $\hat{\Sigma}^j=\tilde{\Sigma}^j/\Z_n$ are smooth, properly embedded surfaces in $M$.  
We claim that the sequence $\{\hat{\Sigma}^j\}_{j\geq J_\varepsilon}$ converges in the sense of varifolds to $\hat{\fbmsgO}_n$ as $j\to\infty$ and is $\Z_2$-almost minimizing, where $\Z_2$ is the action of the symmetry group $\dih_n$ reduced to the regularization $M$ of the quotient $\B^3/\Z_n$.
Indeed, since $\tilde\Sigma^j$ does not intersect the singular locus of $\Z_n$, for any annulus $\widehat{\mathrm{An}}$ with sufficiently small radius in $M$, there exists an annulus $\mathrm{An}$ in $\B^3$ such that $\tilde\Sigma^j\cap\mathrm{An}$ is isometric to $\hat\Sigma^j\cap\widehat{\mathrm{An}}$. 
The equivariant surgery from $\Sigma^j$ to $\tilde\Sigma^j$ is performed in $U_{2\varepsilon}\fbmsgO_n$, and thus away from the singular locus, so it descends to the quotient.
Since $\Sigma^j/\Z_n$ is a topological disc by construction of the sweepout, Lemma~\ref{lem:orientable_surgery} implies $\bsum(\hat{\Sigma}^j)=0$ for all~$j$.
Theorem~\ref{thm:TopologyOfAlmMinSeq} then yields the contradiction
\begin{align*}
1=\bsum(\hat{Z}_n)\leq\liminf_{j\to\infty}\bsum(\hat{\Sigma}^j)=0. 
\end{align*}
Consequently, $\fbmsgO_n$ has exactly $n$ pairwise isometric boundary components and by Lemma \ref{lem:structural}, the cyclic group $\Z_n$ acts simply transitively on their collection.

\ref{thm:fbmsg0-Step4} \emph{Asymptotic behaviour.}
Let $d(n)$ denote the maximal distance a point in $\fbmsgO_n$ can have from the flat, equatorial disc. 
In fact, \cite[Theorem~15]{White2016} implies that $d(n)$ is attained by a point $q_0\in\partial\fbmsgO_n$. 
The dihedral symmetry of $\fbmsgO_n$ implies that its $n$ pairwise isometric boundary components must intersect the equator of $\B^3$.  
Otherwise, recalling that $\Z_n$ acts simply transitively on their collection, all $n$ boundary components would be contained in the same hemisphere, contradicting the $\rotation_{x\text{-axis}}^{\pi}$-equivariance of their union.  
The area estimate from above and \cite[Theorem~2]{White2016} imply $\hsd^1(\partial \fbmsgO_n)=2\hsd^2(\fbmsgO_n)<4\pi$ for all $2\leq n\in\N$. 
Hence, every single boundary component of $\partial\fbmsgO_n$ has length less than $4\pi/n$.  
In particular, $q_0$ is contained in a curve of length less than $4\pi/n$ intersecting the equatorial disc. 
As a result, we obtain that $d(n)<4\pi/n\to0$ as $n\to\infty$. 
Consequently, any subsequence of $\{\fbmsgO_n\}_n$ converges in the sense of varifolds to the equatorial disc with multiplicity $m\in\N$. 
It remains to prove $m=2$.

Since $\pi<\hsd^2(\fbmsgO_n)<2\pi$ for all $n$, we necessarily have $m\in\{1,2\}$. 
Suppose a subsequence of $\{\fbmsgO_n\}_n$ converges to the equatorial disc with multiplicity one. 
Then for any $\varepsilon>0$ there exists $n_\varepsilon\in\N$ such that $\hsd^2(\fbmsgO_{n_\varepsilon})<\pi+\varepsilon$. 
Choosing $\varepsilon>0$ as given by \cite[Proposition~2.1]{Ketover2016FBMS}, we obtain that $\fbmsgO_{n_\varepsilon}$ is isometric to the equatorial disc, contradicting the fact that $\fbmsgO_{n_\varepsilon}$ has $n_\varepsilon$ boundary components. 
Therefore, any subsequence of $\{\fbmsgO_n\}_n$ converges to the equatorial disc with multiplicity $m=2$. 
\end{proof}

\appendix
\section{Relating topological invariants}
\label{app:homology}
 
Given any topological space $X$, the $k$th Betti number $\beta_k(X)$ is defined as the rank of the $k$th homology group $H_k(X)$ and the Euler characteristic $\chi(X)$ coincides with the alternating sum of the Betti numbers (cf. \cite[Theorem 2.44]{Hatcher2002}). 
In this appendix, we focus on the topological space formed by a compact surface $\Sigma$ and determine how the Euler characteristic $\chi$, the first Betti number $\beta_1$ and the topological complexities $\gsum$ and $\bsum$ defined in Definition \ref{defn:abc} are related.

\begin{proposition} \label{prop:TopCptSurf}
Let $\Sigma$ be a compact, connected surface with genus $g\geq0$ and $b\geq0$ boundary components. 
Then, its first homology group, Betti number and Euler characteristic depend as follows on $g$, $b$ and the orientability. 
\[
\begin{array}{ r | l | l | l  }
\Sigma\quad                       &H_1(\Sigma)        &\beta_1(\Sigma)&\chi(\Sigma)\\ \hline
\text{orientable and closed}      &\Z^{2g}            &2g             &2-2g        \\ 
\text{nonorientable and closed}   &\Z^{g-1}\times\Z_2 &g-1            &2-g         \\  
\text{orientable with boundary}   &\Z^{2g+b-1}        &2g+b-1         &2-2g-b      \\
\text{nonorientable with boundary}&\Z^{g+b-1}         &g+b-1          &2-g-b  
\end{array}
\]
\end{proposition}

\begin{proof} 
We have $\beta_0(\Sigma)=1$ because $\Sigma$ is connected and $\beta_k(\Sigma)=0$ for all $k\geq3$ because $\Sigma$ is two-dimensional. 
Moreover, $\beta_2(\Sigma)=1$ if $\Sigma$ is orientable without boundary and $\beta_2(\Sigma)=0$ otherwise. (Note that if $\Sigma$ has nonempty boundary, then it does not enclose a volume.)  
Hence,  
\begin{align}\label{eqn:chi(Sigma)}
\chi(\Sigma)&=\begin{cases}
                 2-\beta_1(\Sigma) & \text{ if $\Sigma$ is orientable without boundary, }\\
                 1-\beta_1(\Sigma) & \text{ otherwise. }
             \end{cases}
\end{align}						
The computation of $\beta_1(\Sigma)$ and $H_1(\Sigma)$ in the cases where $\Sigma$ is closed can be found e.\,g.~in \cite[Chapter 4 Proposition~5.1]{Massey1977}. 
Given a surface $\Sigma$ with nonempty boundary, let $\hat\Sigma$ be the closed surface which we obtain from $\Sigma$ by closing up each of its $b$ boundary components by gluing in a topological disc.  
Since the Euler characteristic of a surface is the alternating sum of the number of vertices, edges and faces in any suitable triangulation, we have $\chi(\Sigma)=\chi(\hat\Sigma)-b$ and by \eqref{eqn:chi(Sigma)}
\begin{align}
\notag
\beta_1(\Sigma)=1-\chi(\Sigma)=1-\chi(\hat\Sigma)+b
&=\begin{cases}
2g+b-1& \text{ if $\Sigma$ is orientable, }\\
g+b-1& \text{ if $\Sigma$ is nonorientable. }
\end{cases}
\end{align}
This completes the computation of the Euler characteristic and the rank of the first homology group of a surface with boundary. 
The homology group itself can be determined as follows. 
An orientable surface $\Sigma$ with genus $g\in\N$ and $1\leq b\in\N$ boundary components is homeomorphic to a $4g$-gon, where the edges are identified according to the cyclic labeling \[a_1,b_1,a_1^{-1},b_1^{-1},\ldots,a_g,b_g,a_g^{-1},b_g^{-1}\] 
(see \cite[Chapter~4 Section~5.3]{Massey1977}) and where $b$ pairwise disjoint topological discs are removed from the interior.  
After the identification, the curves $a_1,b_1,a_2,b_2,\ldots,a_g,b_g$ are closed and have a point $x_0\in\Sigma$ in common.
Let $c_1,\ldots,c_{b-1}$ be closed curves based at $x_0$ and going around $b-1$ out of the $b$ boundary components. 
Then, the union of the $2g+b-1$ curves $a_1,b_1,\ldots,a_g,b_g,c_1,\ldots,c_{b-1}$ is a deformation retract of $\Sigma$. 
This implies that $H_1(\Sigma)$ is the free $\Z$-module of rank $2g+b-1$, namely $H_1(\Sigma)\cong \Z^{2g+b-1}$.

In the nonorientable case, $\Sigma$ is homeomorphic to a $2g$-gon, where the edges are identified as in 
\cite[Chapter~4 Section~5.4]{Massey1977} and where $b$ pairwise disjoint topological discs are removed. 
The conclusion that $H_1(\Sigma)\cong \Z^{g+b-1}$ then follows similarly as in the orientable case. 
\end{proof}

\begin{corollary}\label{cor:Eulerabc}\label{cor:beta=2g}
Let $\Sigma$ be any compact, possibly disconnected surface and let $c_{\mathcal{O}}$ respectively $c_{\mathcal{N}}$ be the number of its orientable respectively nonorientable connected components.  
Recalling the notation introduced in Definition \ref{defn:abc}, the Euler characteristic of $\Sigma$ is equal to 
\[
\chi(\Sigma)=2c_{\mathcal{O}}+c_{\mathcal{N}}-2\gsum(\Sigma)-\beta_0(\partial\Sigma).
\]
\end{corollary}

\begin{proof}
For every orientable connected component $\hat\Sigma$ of $\Sigma$, Proposition \ref{prop:TopCptSurf} implies
\begin{align}\label{eqn:20230412-1}
\chi(\hat\Sigma)
&=2-2\genus(\hat\Sigma)-\beta_0(\partial\hat\Sigma)
=2-2\gsum(\hat\Sigma)-\beta_0(\partial\hat\Sigma), 
\intertext{and for every nonorientable connected component $\hat\Sigma$ of $\Sigma$, we have}
\label{eqn:20230412-2}
\chi(\hat\Sigma)
&=2-\genus(\hat\Sigma)-\beta_0(\partial\hat\Sigma)
=1-2\gsum(\hat\Sigma)-\beta_0(\partial\hat\Sigma).
\end{align}
Since the Euler characteristic $\chi$, the genus complexity $\gsum$ and the number of boundary components are all additive with respect to taking the union of different connected components, the claim follows from \eqref{eqn:20230412-1} and \eqref{eqn:20230412-2} by summation. 
\end{proof}

\begin{corollary}\label{cor:beta1=2g+b+c}
Let $\Sigma$ be any compact, possibly disconnected surface. 
Then its first Betti number $\beta_1(\Sigma)$ coincides with the sum of $2\gsum(\Sigma)+\bsum(\Sigma)$ and the number of its nonorientable connected components with nonempty boundary. 
\end{corollary}

\begin{proof}
Given any connected component $\hat\Sigma$ of $\Sigma$, Proposition \ref{prop:TopCptSurf} and Definition \ref{defn:abc} imply: 
\begin{itemize}[nosep]
\item If $\hat\Sigma$ is orientable and closed then 
$\beta_1(\hat\Sigma)=2\genus(\hat\Sigma)=2\gsum(\hat\Sigma)=2\gsum(\hat\Sigma)+\bsum(\hat\Sigma)$. 
\item If $\hat\Sigma$ is orientable with boundary then 
$\beta_1(\hat\Sigma)=2\genus(\hat\Sigma)+\beta_0(\partial\hat\Sigma)-1=2\gsum(\hat\Sigma)+\bsum(\hat\Sigma)$. 
\item If $\hat\Sigma$ is nonorientable and closed then 
$\beta_1(\hat\Sigma)=\genus(\hat\Sigma)-1=2\gsum(\hat\Sigma)=2\gsum(\hat\Sigma)+\bsum(\hat\Sigma)$. 
\item If $\hat\Sigma$ is nonorientable with boundary then 
$\beta_1(\hat\Sigma)=\genus(\hat\Sigma)+\beta_0(\partial\hat\Sigma)-1=2\gsum(\hat\Sigma)+\bsum(\hat\Sigma)+1$. 
\end{itemize}
As in the proof of Corollary \ref{cor:beta=2g}, the claim follows by additivity with respect to taking the union of different connected components. 
\end{proof}

\begin{corollary}\label{cor:homology2gsum}
Given any compact surface $\Sigma$ with nonempty boundary, let $\iota\colon H_1(\partial \Sigma)\to H_1(\Sigma)$ be the map induced by the inclusion $\partial\Sigma\hookrightarrow\Sigma$. 
Then the quotient $H_1(\Sigma)/\iota(H_1(\partial\Sigma))$ has rank $2\gsum(\Sigma)$.
\end{corollary}

\begin{proof}
Without loss of generality we may assume that $\Sigma$ is connected.
Let $g\geq0$ and $b\geq1$ be the genus and the number of boundary components of $\Sigma$ respectively. 
By \cite[Theorem~2.16]{Hatcher2002}, the relative homology groups $H_k(\Sigma,\partial\Sigma)$ fit into the long exact sequence
\begin{equation}\label{eqn:relativehomology}
\begin{tikzcd}[row sep=tiny,column sep=small]
H_2(\Sigma) \arrow[r] &   
H_2(\Sigma,\partial\Sigma) \arrow[r,hook] & 
H_1(\partial\Sigma)\arrow[r,"\displaystyle\iota"] &  
H_1(\Sigma) \arrow[r] &   
H_1(\Sigma,\partial\Sigma) \arrow[r] &
H_0(\partial\Sigma)  \arrow[r] &  
H_0(\Sigma).
\end{tikzcd}
\end{equation}
By assumption, $H_2(\Sigma)\cong0$ and $H_1(\partial\Sigma)\cong\Z^b$. 
If $\Sigma$ is orientable, then $H_2(\Sigma,\partial\Sigma)\cong\Z$. 
Since \eqref{eqn:relativehomology} is exact, 
the kernel of $\iota$ has rank $1$ and the image of $\iota$ has rank $b-1$. 
Recalling $H_1(\Sigma)\cong\Z^{2g+b-1}$ from Proposition~\ref{prop:TopCptSurf}, we obtain 
that $H_1(\Sigma)/\iota(H_1(\partial\Sigma))$ has rank $2g$. 

If $\Sigma$ is nonorientable, then $H_2(\Sigma,\partial\Sigma)\cong0$ and $\iota$ is injective. 
In this case, $H_1(\Sigma)\cong\Z^{g+b-1}$ by Proposition \ref{prop:TopCptSurf}, and we obtain that $H_1(\Sigma)/\iota(H_1(\partial\Sigma))$ has rank $g-1$. 
In either case, the rank of the quotient coincides with $2\gsum(\Sigma)$ by Definition \ref{defn:abc}. 
\end{proof}


\setlength{\parskip}{.9ex plus 1pt minus 2pt}
\bibliography{fbms-bibtex}

\printaddress

\end{document}